\documentclass{amsart}

\newtheorem{theorem}{Theorem}[section]
\newtheorem{lemma}[theorem]{Lemma}
\newtheorem{proposition}[theorem]{Proposition}
\theoremstyle{definition}

\theoremstyle{remark}
\newtheorem{remark}[theorem]{Remark}

\numberwithin{equation}{section}

\usepackage{amsmath}
\usepackage{lmodern}  

\usepackage{amsfonts,amssymb}
\usepackage{epsfig}
\usepackage{booktabs}

\usepackage{xcolor}
\usepackage{mathrsfs}
\graphicspath{{figure/}}
\usepackage{float}
\usepackage{nicematrix}
\usepackage{mathabx}

 \def\iu{\mathrm{i}}
\def\eps{\varepsilon}
\def\d{\mathrm{d}}
\def\e{{\mathrm e}}

\def\wt{\widetilde}
\def\wh{\widehat}

\def\ecl{\color{black}}

\usepackage{cancel}

\begin{document}

\title[Weighted finite difference methods for semiclassical NLS]{Weighted finite difference methods for the semiclassical nonlinear Schr\"odinger equation with multiphase oscillatory initial data}

\author[Y.\ Shi]{Yanyan Shi}
\address{Mathematisches Institut, Univ.\ T\"ubingen, D-72076 T\"ubingen, Germany}
\curraddr{}
\email{shi@na.uni-tuebingen.de}
\thanks{}

\author[Ch.\ Lubich]{Christian Lubich}
\address{Mathematisches Institut, Univ.\ T\"ubingen, D-72076 T\"ubingen, Germany}
\curraddr{}
\email{lubich@na.uni-tuebingen.de}
\thanks{This work was supported by the Deutsche Forschungsgemeinschaft (DFG, German Research Foundation) -- Project-ID258734477 -- SFB 1173.}


\date{}

\dedicatory{}

\begin{abstract}
This paper introduces weighted finite difference methods for numerically solving dispersive evolution equations with solutions that are highly oscillatory in both space and time. We consider a
semiclassically scaled cubic nonlinear Schr\"odinger equation with  highly oscillatory initial data, first in the single-phase case and then in the general multiphase case. The proposed methods do not need to resolve high-frequency oscillations in both space and time by prohibitively fine grids as would be required by standard finite difference methods.  The approach taken here modifies traditional finite difference methods by  appropriate exponential weights. Specifically, we propose the weighted leapfrog and weighted Crank--Nicolson methods, both of which achieve second-order accuracy with time steps and mesh sizes that are not restricted in magnitude by the small semiclassical parameter. Numerical experiments illustrate the theoretical results.
\bigskip


\bigskip

\end{abstract}

\maketitle



\section{Introduction}
As a basic model problem of a dispersive evolution equation with solutions that are highly oscillatory in both space and time, we consider the time-dependent weakly nonlinear Schr\"odinger equation in semiclassical scaling~\cite{carles2008semi,carles2010multiphase},
\begin{equation}\label{eq:schr}
\iu\eps\,\partial_t u+\frac{\eps^2}{2} \Delta u=\eps\lambda  \,|u|^2u,
\end{equation}
where $0<\eps\ll 1$ is the small semiclassical parameter, and $\lambda$ is a fixed nonzero real number. 

This equation is to be solved for the complex-valued function $u=u(t,x)$ 
under periodic boundary conditions with $x\in \mathbb{T}^d=(\mathbb{R}/2\pi\mathbb{Z})^d$ over a bounded time interval ${0\le t\le T}$. The final time $T$ is chosen independently of $\eps$. On this time scale, the nonlinearity has an $\mathcal{O}(1)$ effect on the solution. 

We consider highly oscillatory initial data at $t=0$. A simple model problem is the case of a {\it single phase} (also known as the monochromatic case) considered in the first part of this paper,
\begin{equation}\label{eq:init-1}
u(0,x)=a^0(x)\,\e^{\iu\kappa\cdot x/\eps},
\end{equation}
where $\kappa \in\mathbb{R}^d\backslash \{0\}$ is a  fixed wave vector, and 
$a^0:\mathbb{T}^d\rightarrow \mathbb{C}$ is a given profile function that is assumed to be {\it smooth} in the sense of having arbitrarily many
higher-order derivatives bounded independently of $\eps$. 

The initial function $u(0,\cdot)$ in \eqref{eq:init} is required to be a $2\pi$-periodic continuous function. This is satisfied if
the small parameter $\eps$ is assumed to take only values for which $\kappa/\eps \in \mathbb{Z}^d$ and $a^0$ are $2\pi$-periodic. This assumption on $\eps$ is not a restriction, since it can always be achieved with an 
$\mathcal{O}(\eps)$ modification of $\kappa$ and a corresponding smooth modification of $a^0$.

The single-phase problem \eqref{eq:schr}--\eqref{eq:init-1} provides much insight into the construction and analysis of the proposed numerical methods and so prepares the ground for the numerical treatment of the more challenging {\it multiphase problem} with initial data
\begin{equation}\label{eq:init}
u(0,x)=\sum_{m=1}^M a^0_m(x)\,\e^{\iu\kappa_m\cdot x/\eps},
\end{equation}
where $\kappa_m \in\mathbb{R}^d\backslash \{0\}$ are  fixed wave vectors, and 
$a_m^0:\mathbb{T}^d\rightarrow \mathbb{C}$ are given smooth profile functions. 
This is considered in the second part of the paper, first illustrated by the two-phase case $M=2$ with opposite wave numbers $\kappa_1=\kappa$ and $\kappa_2=-\kappa$ and then extended to the general multiphase case. 

An analytical study of the multiphase problem was given by Carles, Dumas \& Sparber \cite{carles2010multiphase}, where $\mathcal{O}(\eps)$ approximations to the solution are constructed that are of the form \eqref{eq:init} at every time. Here we prove a refined result (Theorem~\ref{th:MFE}) that provides a second-order expansion of the solution with an $\mathcal{O}(\eps^2)$ error in the maximum norm. As in \cite{carles2010multiphase}, we derive error bounds in the stronger norm of the Wiener algebra, which is particularly suitable for handling the nonlinearity. The improved analytical approximation result is basic for the numerical analysis.

The solution $u(t,x)$ is highly oscillatory in both time and space at a scale proportional to the small parameter $\eps$. This poses significant challenges in the development of efficient numerical methods and their error analysis. 
Traditional finite difference methods like the leapfrog and Crank--Nicolson schemes have been studied for Schr\"odinger equations in the semiclassical scaling in~\cite{markowich1999numerical}, where stringent restrictions on the time step $\tau\ll\eps$ and mesh size $h\ll\eps$ are required. Time-splitting spectral discretizations, also known as split-step Fourier methods~\cite{bao2002time,bao2003numerical,jin2011mathematical,lasser2020computing}, ease these restrictions. For the stated initial value problem \eqref{eq:schr}--\eqref{eq:init-1}, using techniques as in the cited papers, split-step Fourier methods can be shown to require no bound of $\tau$ in terms of $\eps$, but they still require small $h=\mathcal{O}(\eps/|\log \eps|)$ to obtain at least first-order accuracy in $h$, as is already needed for the approximation of the initial data by trigonometric interpolation.
Asymptotic-preserving methods have been proposed in~\cite{arnold2011wkb,degond2007asymptotic,besse2013asymptotic} by  reformulating the Schr\"odinger equation  using the WKB expansion~\cite{grenier1998semiclassical,carles2008semi} or the Madelung transform~\cite{madelung1927quantum}. 

One objective of this paper is to revive finite difference methods for dispersive evolution equations with solutions that are highly oscillatory in both space and time, modifying standard methods such as leapfrog and Crank--Nicolson methods by changing the method coefficients on the same stencil.
{\it Such weighted schemes enable us to approximate the solution of \eqref{eq:schr}--\eqref{eq:init} with second-order accuracy even when using time steps $\tau$ and mesh sizes $h$ that are not restricted by $\eps$.} Under mild assumptions, we prove an $\mathcal{O}(\tau^2+h^2+\eps^2)$ error bound for the multiphase problem (Theorem~\ref{th:conv-m}).
The proposed weighted methods tend to the standard leapfrog and Crank--Nicolson schemes as the ratios of the time step and mesh size to the semiclassical parameter $\eps$ approach zero, which is, however, not the regime of principal interest in this paper. 
The methods can be extended to be not only asymptotic-preserving as $\eps\to 0$ but also uniformly accurate (of order $4/5$ for the multiphase problem) for $0<\eps\le 1$.

Modulated Fourier expansions are a powerful tool for deriving and analyzing numerical methods for highly oscillatory problems. They represent both the exact and the numerical solution as sums of products of slowly varying modulation functions and highly oscillatory exponentials, as is given here with the initial data \eqref{eq:init}.  Comparing the modulated Fourier expansions of the numerical and the exact solution then yields asymptotically sharp error bounds; see~\cite[Chapter XIII]{hairer02gni} and, e.g., \cite{hairer22large} for oscillatory ordinary differential equations and, e.g.,
\cite{cohen08coe,faou14pws} for evolutionary partial differential equations. We will pursue such an approach also here, as a novelty combined in both time and space.  For the single-phase problem and in some cases, depending on the numerical treatment of the nonlinearity, also for the multiphase problem, the proposed weighted finite difference methods can be reinterpreted as applying the corresponding standard finite difference schemes to the equations for the non-oscillatory modulation functions of the modulated Fourier expansion (here considered up to order 2). The approach of numerically approximating the modulation functions has previously been used in the literature for oscillatory ordinary differential equations, e.g., in \cite{cohen04diss,condon09oho}, and later for temporally (though not spatially) oscillatory partial differential equations, e.g., in \cite{bao14aua,faou14aps}.

We will formulate the weighted finite difference methods 
only in the spatially one-dimensional case ($d=1$). This apparent limitation is introduced only for ease of presentation. The methods and theoretical results extend directly to higher dimensions. Furthermore, the extension to the full space $\mathbb{R}^d$ instead of the torus $\mathbb{T}^d$ is straightforward for the formulation of the methods and can be done analogously in the theory. 

The paper is organized as follows.

In Section \ref{sec:algo}, we introduce the weighted leapfrog and weighted Crank–Nicolson algorithms for a single initial phase, with stepsizes $\tau$ and meshwidths $h$ that can be arbitrarily large compared to $\eps$. For $h \gg \eps$, there is a mild stepsize restriction $\tau \le c h$ for the weighted leapfrog method and no such restriction for the weighted Crank–Nicolson method.

In Section \ref{sec:error}, Theorem~\ref{th:conv-1} states $\eps$-uniform $O(\tau^2+h^2)$ error bounds for the single-phase case for both numerical methods. Numerical experiments confirm these theoretical results. The proof of the error bound is provided in Sections \ref{sec:consistency} and \ref{sec:stability}.
In Section \ref{sec:consistency}, we study the consistency error, i.e., the defect obtained on inserting the exact solution into the numerical scheme. Section \ref{sec:stability} presents the linear Fourier stability analysis, which is done in the Wiener algebra $A(\mathbb{T}) \subset C(\mathbb{T})$, and then gives a nonlinear stability analysis that bounds the error
of the numerical solution in terms of the defect.

In Section \ref{sec:two}, we treat the case of two opposite phases. We formulate the modulated Fourier expansion of the exact solution and extend both the weighted leapfrog and weighted Crank–Nicolson methods to the two-phase case in several variants that differ in the treatment of the nonlinearity and in the attained accuracy. For the most accurate variant, we have an $\mathcal{O}(\tau^2+h^2+\eps^2)$ error bound. Numerical experiments illustrate the theory. 

In Section \ref{sec:multi}, we present the modulated Fourier expansion of the exact solution for general multiphase initial conditions \eqref{eq:init}, and we prove that the remainder term is of order $\mathcal{O}(\eps^2)$ in the maximum norm (Theorem~\ref{th:MFE}). 

In Section \ref{sec:multi-scheme}, we extend the weighted leapfrog and Crank-Nicolson methods to the general multiphase setting and prove an  $\mathcal{O}(\tau^2+h^2+\eps^2)$ error bound in Theorem~\ref{th:conv-m}, based on Theorems~\ref{th:conv-1} and~\ref{th:MFE}. We further obtain $\eps$-uniform convergence of order $4/5$ by combining the above error bound  with the standard error bound obtained from Taylor expansion of the solution, which is $\mathcal{O}{(\tau^2+h^2)/\eps^3}$. The latter bound is smaller only for $\tau^2 + h^2\lesssim \eps^5$, which is not the situation of main interest here, where we aim for large stepsizes $\tau\gg\eps$ and meshsizes $h\gg\eps$ for small $\eps$.

\bigskip
\centerline{\bf Part I. Case of a single phase}

\section{Weighted finite difference methods for a single initial phase}\label{sec:algo}
For simplicity of presentation, we restrict the presentation to the case of one spatial dimension, $0\leq x\leq 2\pi$, with periodic boundary conditions. The proposed numerical methods and their analysis extend to higher dimensions in a straightforward way.

\subsection{Preparation: Weighted finite differences of modulated exponentials}
\label{subsec:prep}
We expect that the solution to \eqref{eq:schr} with initial data \eqref{eq:init-1} can be approximated by a modulated plane wave
$$
v(t,x) = b(t,x)\, \e^{\iu(\kappa x-\omega t)/\eps},
\qquad\text{where}\quad
\omega = \tfrac12\kappa^2
$$
in view of the dispersion relation $\iu\eps(-\iu \omega/\eps) + \tfrac12\eps^2 (\iu\kappa/\eps)^2=0$
of the free linear Schr\"odinger equation, and $b(t,x)$ is a smooth modulation function with derivatives bounded independently of $\eps$. We then have
\begin{align*}
    \partial_t v(t,x) &= \Bigl( \partial_t - \frac{\iu\omega}\eps \Bigr) b(t,x) \cdot \e^{\iu(\kappa x-\omega t)/\eps},
    \\
    \partial_x^2 v(t,x) &= \Bigl( \partial_x + \frac{\iu\kappa}\eps \Bigr)^2 b(t,x) \cdot \e^{\iu(\kappa x-\omega t)/\eps}.
\end{align*}
We approximate the partial derivatives of $b$ by symmetric finite differences, with a temporal step size $\tau$ and a spatial grid size $h$, up to errors of $\mathcal{O}(\tau^2)$ and $\mathcal{O}(h^2)$ resulting from the Taylor expansion of the smooth function $b$ at $(t,x)$,
\begin{align*}
    \partial_t v(t,x) &\approx \biggl( \frac{b(t+\tau,x)-b(t-\tau,x)}{2\tau} - 
    \frac{\iu\omega}\eps \,b(t,x) \biggr)\e^{\iu(\kappa x-\omega t)/\eps}
    \\
    &= \frac{\e^{\iu\omega\tau/\eps} v(t+\tau,x)- \e^{-\iu\omega\tau/\eps} v(t-\tau,x)}{2\tau} - 
    \frac{\iu\omega}\eps \,v(t,x)
\end{align*}
and 
\begin{align*}
    &\partial_x^2 v(t,x) \approx
    \biggl( \frac{b(t,x+h)-2b(t,x)+b(t,x-h)}{h^2} + 
    2 \frac{\iu\kappa}\eps \,\frac{b(t,x+h)-b(t,x-h)}{2h} \biggr.
    \\
    &\qquad\qquad\qquad \biggl. - \frac{\kappa^2}{\eps^2}\, b(t,x)\biggr)
    \e^{\iu(\kappa x-\omega t)/\eps} =
    \\
    &\frac{(1+\iu \kappa h/\eps)\,\e^{-\iu \kappa h/\eps} v(t,x+h)- 2v(t,x) +
    (1-\iu \kappa h/\eps)\,\e^{\iu \kappa h/\eps} v(t,x+h)}{h^2} - 
    \frac{\kappa^2}{\eps^2} \,v(t,x).
\end{align*}
We will use the so obtained exponentially weighted finited differences in the numerical schemes to be proposed next. We further note that up to an $\mathcal{O}(\tau^2)$ error,
\begin{align*}
v(t,x) &= b(t,x)\, \e^{\iu(\kappa x-\omega t)/\eps} \approx 
\tfrac12 \bigl(b(t+\tau,x)+b(t-\tau,x)\bigr)\, \e^{\iu(\kappa x-\omega t)/\eps} 
\\
&= \tfrac12 \bigl(\e^{\iu\omega\tau/\eps} v(t+\tau,x)+\e^{-\iu\omega\tau/\eps}v(t-\tau,x)\bigr).
\end{align*}

\subsection{Exponentially weighted leapfrog algorithm.} 
\label{subsec:wlf}
Let the time step be $\tau=T/N>0$  and the mesh size $h=2\pi/M>0$, where $N$ and $M$ are positive integers. We denote by
$u^{n}_{j}$ the numerical approximation of $u(t_n,x_j)$, where $t_n=n\tau$ for $0\leq n\leq N$, and $x_j=jh$ for $0\leq j\leq M$. Using weighted finite differences as derived above, we introduce an explicit algorithm, which has the symmetric two-step formulation
\begin{equation}\label{eq:scheme}
\begin{aligned}
\iu\eps\,\frac{\e^{\iu\alpha}u_j^{n+1}-\e^{-\iu\alpha}u_j^{n-1}}{2\tau}
+\frac{\eps^2}{2}\, \frac{\e^{-\iu\beta}(1+\iu\beta)u^n_{j+1}-2u^n_j+\e^{\iu\beta}(1-\iu\beta)u^n_{j-1}}{h^2}\\
=\eps\lambda|u^n_j|^2u^n_j
\end{aligned}
\end{equation}
with 
\[
\alpha=\frac{\omega\tau}\eps, \quad \beta=\frac{\kappa h}\eps.
\]
Note that the terms $\omega u^n_j$ and $-\tfrac12 \kappa^2 u^n_j$, which would appear in the weighted finite difference approximations to $\iu\eps \partial_t u(t_n,x_j)$ and $\tfrac12 \eps^2 \partial_x^2 u(t_n,x_j)$, respectively, cancel thanks to the dispersion relation $\omega=\tfrac12 \kappa^2$.

The weighted leapfrog scheme tends to the classical leapfrog scheme in the limit $\tau/\eps\to 0$ and $h/\eps\to 0$. Our main interest here is, however, to use the weighted scheme with large ratios $\tau/\eps$ and $h/\eps$. 

For the weighted leapfrog method we need the following CFL-type condition. 

\medskip\noindent
{\it Stability condition:} 
    \begin{equation}\label{eq:stability}
        \eps\tau < h^2 /\gamma 
        \qquad\text{with}\quad \gamma= \gamma(\beta)= 1+\max(|\beta|,1). 
    \end{equation}
    
    \noindent
    Equivalently, $\alpha/\beta^2 <1/(2\gamma)$. 
    For large $\beta$ we note
    $1/\gamma \approx 1/|\beta|= \eps/|\kappa h|$. This yields the condition $\tau <  h/|\kappa|$,
    which is the CFL condition for the advection equation $\partial_t a + \kappa \partial_x a =0$. 
    
    On the other hand, for small $|\beta|$, \eqref{eq:stability} becomes the CFL condition $\eps\tau< \tfrac12 h^2$ of the classical unweighted leapfrog method applied to \eqref{eq:schr}, 
    which in our highly oscillatory situation requires in addition $\tau\ll\eps$ and $h\ll\eps$ to have a small consistency error.

    As a starting step, we use a step of the weighted explicit Euler method
\begin{equation}\label{eq:scheme-starting}
\begin{aligned}
\iu\eps\,\frac{\e^{\iu\alpha}u_j^{1}-u_j^{0}}{\tau}
+\frac{\eps^2}{2}\, \frac{\e^{-\iu\beta}(1+\iu\beta)u^0_{j+1}-2u^0_j+\e^{\iu\beta}(1-\iu\beta)u^0_{j-1}}{h^2}
=\eps\lambda|u^0_j|^2u^0_j,
\end{aligned}
\end{equation}
with initial data $u^0_j=u(0,x_j)$ given by \eqref{eq:init-1}.
\subsection{Exponentially weighted Crank--Nicolson algorithm.} 
\label{subsec:wcn}
We further present the following implicit scheme:
    \begin{equation}\label{eq:CN}
    \begin{aligned}
\iu\eps\,\frac{\e^{\iu\alpha}u_j^{n+1}-\e^{-\iu\alpha}u_j^{n-1}}{2\tau}
+\frac{\eps^2}{2}\, \frac{\e^{-\iu\beta}(1+\iu\beta)\tilde{u}^n_{j+1}-2\tilde{u}^n_j+\e^{\iu\beta}(1-\iu\beta)\tilde{u}^n_{j-1}}{h^2}\\
=\eps\lambda\frac{(|u^{n-1}_j|^2+|u^{n+1}_j|^2)\tilde{u}^n_j}{2}
\end{aligned}
    \end{equation}
    with $\tilde{u}^n_j={(\e^{\iu\alpha}u^{n+1}_j+\e^{-\iu\alpha}u^{n-1}_j)}/{2}$. 
Scheme \eqref{eq:CN} implicitly gives the map $u^{n-1}\mapsto u^{n+1}$, since $u^n$ does not appear; using half the time step $\tau\rightarrow{\tau/2}$,  it can be written and implemented as a one-step method $u^n\mapsto u^{n+1}$.

Note that as $\tau/\eps\to 0$ and $h/\eps\to 0$, this scheme tends to the classical Crank--Nicolson scheme. We are, however, interested in using the weighted scheme with large ratios $\tau/\eps$ and $h/\eps$.

No stability condition is needed for the weighted Crank--Nicolson algorithm.
    
\ecl

\section{Error bound and numerical experiments}\label{sec:error}
Writing the exact solution of \eqref{eq:schr} as
\begin{equation} \label{eq:ua}
u(t,x)=a(t,x)\,\e^{\iu (\kappa x-\omega t)/ \eps} \quad\text{ with }\ \omega=\tfrac12 \kappa^2,
\end{equation}
we find, on inserting this function $u$ into the Schrödinger equation \eqref{eq:schr}, that $a(t,x)$ solves the advected nonlinear Schr\"odinger equation 
 \begin{equation}\label{eq:a}
    \partial_t a+ \kappa\,\partial_x a -\frac{\iu\eps}{2}\partial_x^2 a = -\iu\lambda |a|^2a, \quad\ a(0,x)=a^0(x),
    \end{equation}
with initial data $a^0$ that are assumed to be smooth in the sense of having arbitrarily many partial derivatives bounded independently of $\eps$.
By standard arguments, the solution $a(t,x)$ of \eqref{eq:a} is then also smooth on any closed time interval $0\le t \le T$ with $T$ smaller than a possible blowup time.

Our first main result  shows that the dominant oscillatory term of the numerical solution of \eqref{eq:scheme} and \eqref{eq:CN} is the same as for the exact solution, and it provides a second-order error bound in the maximum norm that is uniform in $\eps$.


\begin{theorem}[$\eps$-uniform second-order convergence in the maximum norm]\label{th:conv-1}
    Let $u^n_j$ be the numerical solution obtained by applying the weighted leapfrog algorithm \eqref{eq:scheme} under the stability condition \eqref{eq:stability} or by the weighted Crank--Nicolson method \eqref{eq:CN} without requiring a stability condition. Assume \eqref{eq:ua} with $a\in C^4([0,T]\times \mathbb{T})$ having fourth-order partial derivatives bounded independently of $\eps$.
 Then, the numerical solution $u^n_j$ can be written as 
    \[
    u^n_j=a(t_n,x_j)\,\e^{\iu (\kappa x_j-\omega t_n)/ \eps}+e^n_j = u(t_n,x_j)+e^n_j
    \]
    for $t_n=n\tau\le T$, $x_j=jh$, where
    $a(t,x)$ is the solution of \eqref{eq:a} and
    the error is bounded in the maximum norm by
    \[
    \max_{n,j} |e^n_j | \leq C (\tau^2+h^2).
    \]
    Here, $C$ is independent of $\tau, h$ and $0<\eps\le 1$, but depends on the final time $T$ and on
    $\theta=\gamma\eps\tau/h^2<1$ with $\gamma$ of \eqref{eq:stability} in the case of the weighted leapfrog method.
\end{theorem}

The proof will be given in the following two sections.

\begin{remark}
    The numerical scheme yields approximations to the oscillatory solution only at the grid points, with many oscillations between neighboring grid points when $\tau\gg\eps$ or $h\gg\eps$. An interpolant capturing these oscillations is readily obtained by interpolating the values 
    $a^n_j=u^n_j \e^{-\iu (\kappa x_j - \omega t_n)/\eps}$ that are $\mathcal{O}(\tau^2+h^2)$ approximations to the grid values of the smooth function $a(t,x)=u(t,x)\e^{-\iu (\kappa x - \omega t)/\eps}$.
\end{remark}

\begin{remark}
    The weighted leapfrog method \eqref{eq:scheme}
    for the Schrödinger equation \eqref{eq:schr} with single-phase oscillatory initial data \eqref{eq:init-1} turns out to be equivalent to applying the standard leapfrog method to the initial value problem \eqref{eq:a} with smooth initial data $a^0$. A possible proof of Theorem~\ref{th:conv-1} could be based on this observation. In the proof given below, we will not directly use this interpretation, since it does not simplify the analysis. It is, however, helpful to have both interpretations as approximations to both $u$ and $a$ in mind, as will become evident in the multiphase case. The same remark applies to the weighted Crank--Nicolson method \eqref{eq:CN}.
\end{remark}



\noindent{\bf{Numerical experiments.}}
In this numerical test, we consider the one-dimensional semiclassical nonlinear Schr\"odinger equation
\[
\iu\eps\partial_t u+\frac{\eps^2}{2}\partial_{xx}u= \eps |u|^2u
\]
with the initial condition
\[
 u(0,x)=\e^{-x^2}\e^{\iu x/\eps}.
\]
We set the spatial domain to $x \in [-6, 6]$ with periodic boundary conditions. The numerical error is measured at the final time $T = 0.5$ using the discrete $L^\infty$ norm over the domain $[-6, 6]$.
 
\begin{figure}[h]
\centerline{
\includegraphics[scale=0.5]{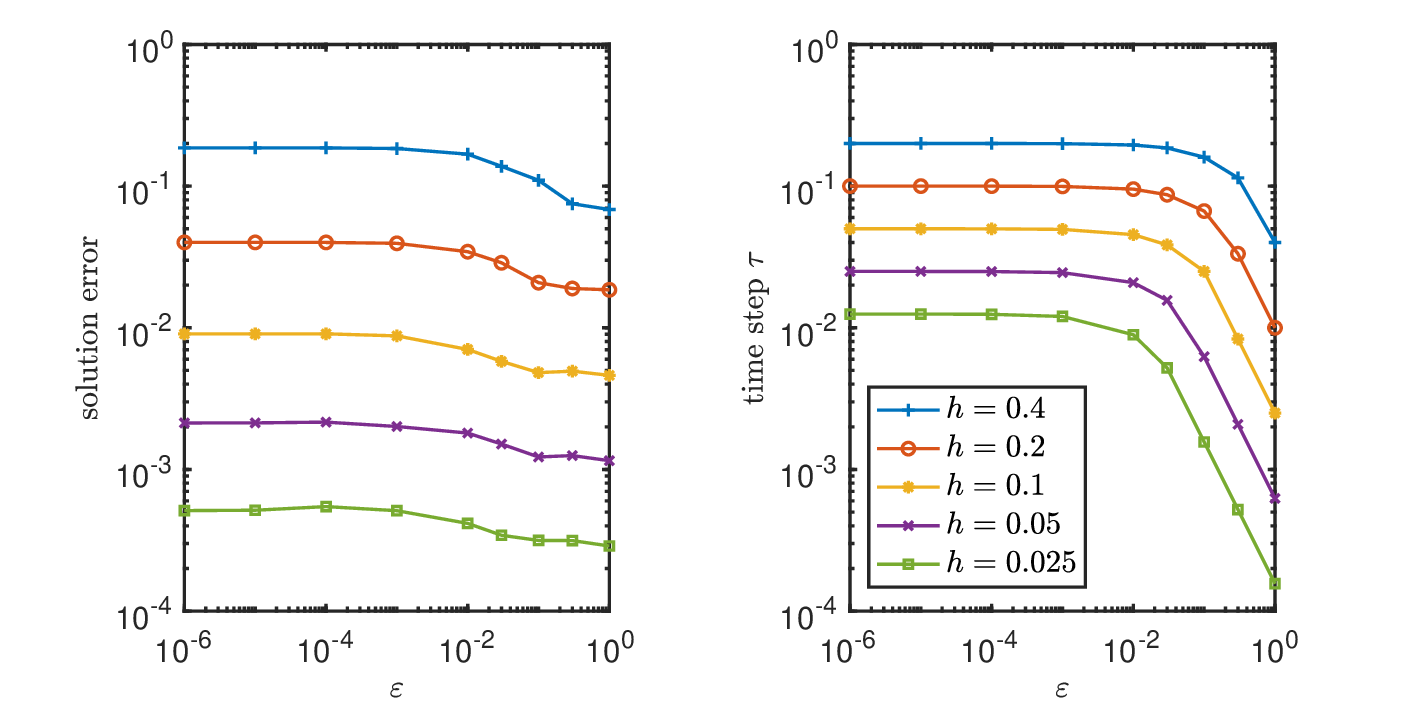}}
\centerline{
\includegraphics[scale=0.5]{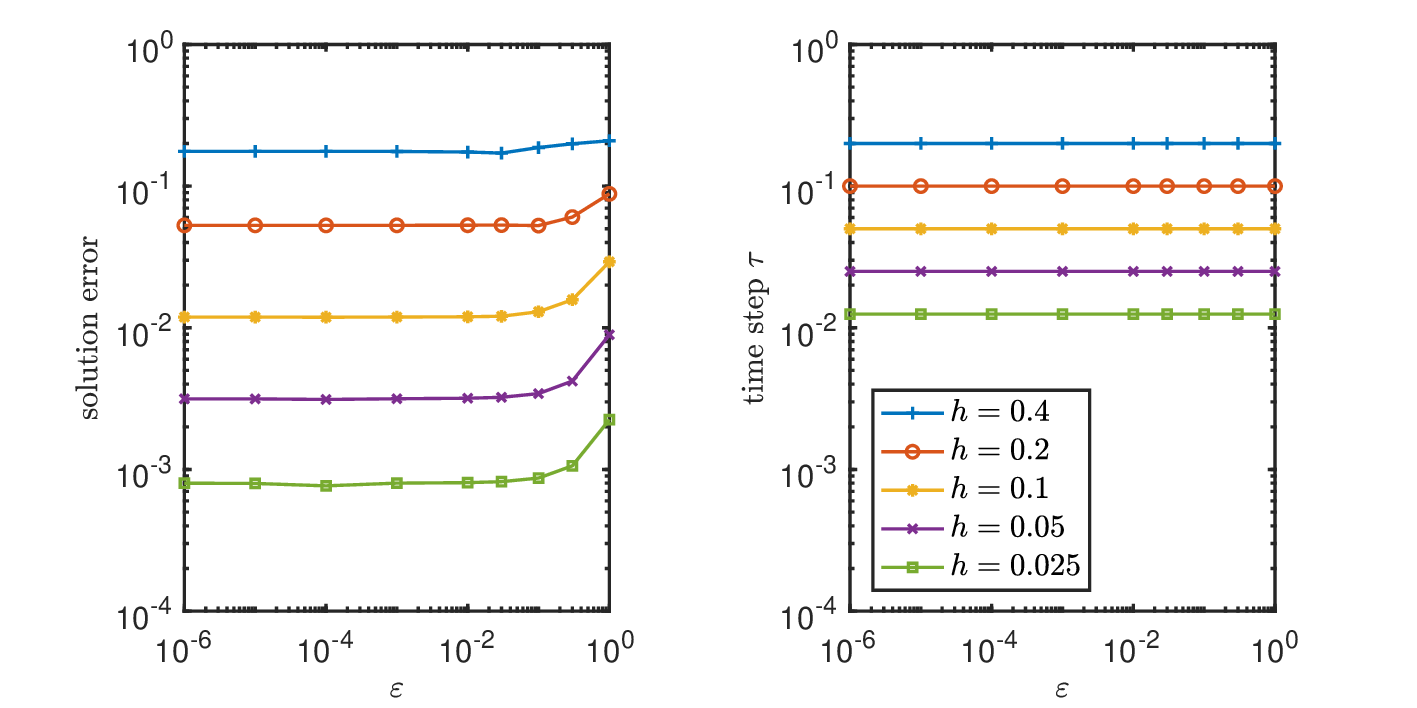}
}
\caption{Error and time stepsize vs.~$\eps$ with different $h$ for the weighted leapfrog method (top row) and weighted Crank-Nicolson method (bottom row).}
\label{fig:flf}
\end{figure}

 Figure~\ref{fig:flf} displays the absolute error in $u$ at time $T = 0.5$ plotted against $\eps$ for several fixed values of $h$. The top row corresponds to results obtained using the weighted leapfrog method. For all values of $\eps$, the error levels off at a value proportional to $h^2$. The corresponding time step $\tau$ as a function of $\eps$ is shown in the right panel. We 
 chose $\tau=\min(h/2,h^2/(2\eps\gamma))$ so that the stability condition \eqref{eq:stability} is satisfied for all $h$ and $\eps$.
 We observe that for large $\eps$, the time step $\tau$ scales with $h^2$, whereas for small~$\eps$, a linear dependence between $\tau$ and $h$ suffices. Similarly, we test the weighted Crank--Nicolson method with time step $\tau=h/2$ for all values of $\eps$, and the results, presented in the bottom row of Figure~\ref{fig:flf}, are consistent with the theoretical prediction in Theorem~\ref{th:conv-1}.

\section{Consistency}
\label{sec:consistency}
We consider the defect obtained on inserting $u(t,x)=a(t,x)\e^{\iu(\kappa x-\kappa^2t/2)/\eps}$ into the weighted leapfrog scheme \eqref{eq:scheme}, 
\begin{equation}
\begin{aligned}
d(t,x) := &\ \iu\eps\,\frac{\e^{\iu\alpha}u(t+\tau,x)-\e^{-\iu\alpha}u(t-\tau,x)}{2\tau}
\\&
+\ \frac{\eps^2}{2}\, \frac{\e^{-\iu\beta}(1+\iu\beta)u(t,x+h)-2u(t,x)+\e^{\iu\beta}(1-\iu\beta)u(t,x-h)}{h^2}\\
&-\ \eps\lambda|u(t,x)|^2 u(t,x),
\end{aligned} 
\label{eq:defect}
\end{equation}
again with $\alpha=\omega\tau/\eps$ and $\beta=\kappa h/\eps$.

\subsection{Defect bound in the maximum norm}

\begin{lemma} \label{lem:defect-max}
In the situation of Theorem~\ref{th:conv-1}, the defect \eqref{eq:defect} is bounded in the maximum norm by
$$
\| d \|_{C([0,T]\times \mathbb{T})} \le c\eps(\tau^2 + h^2),
$$
where $c$ is independent of  $\eps$, $\tau$, $h$ and $n$ with $t_n=n\tau\le T$.
\end{lemma}

\begin{proof}  The $\mathcal{O}(\tau^2)$ and $\mathcal{O}(h^2)$ error bounds of the weighted finite differences in Section~\ref{subsec:prep} yield, omitting the omnipresent argument $(t,x)$ on the right-hand side,
\begin{align*}
    d(t,x) = \Bigl(\iu\eps\,\partial_t u+\tfrac12\eps^2 \partial_x^2 u-\eps\lambda  \,|u|^2u \Bigr)
 - \Bigl(\omega u - \tfrac12 \kappa^2 u \Bigr) +O\bigl(\eps(\tau^2+h^2)\bigr).
\end{align*}
The terms in big brackets  vanish by the nonlinear Schrödinger equation \eqref{eq:schr} and the dispersion relation $\omega=\tfrac12 \kappa^2$. This proves the result.
\end{proof}

However, the maximum norm in the defect bound of Lemma~\ref{lem:defect-max} turns out to be too weak a norm for the proof of Theorem~\ref{th:conv-1}. 

\subsection{Defect bound in the Wiener algebra norm}
Let $A(\mathbb{T})$ be the space of $2\pi$-periodic complex-valued functions with absolutely convergent Fourier series $f(x)=\sum_{k=-\infty}^\infty \widehat f(k)\, \e^{\iu kx}$, equipped with the $\ell^1(\mathbb{Z})$ norm of the sequence of Fourier coefficients.  For the pointwise product of two functions $f,g\in A(\mathbb{T})$ we then have (see, e.g., \cite[Section I.6]{katznelson76})
\begin{equation}\label{eq:tri}
\| fg \|_{A(\mathbb{T})} \le \| f \|_{A(\mathbb{T})} \,\| g \|_{A(\mathbb{T})},
\end{equation}
which makes $A(\mathbb{T})$ a Banach algebra, known as the Wiener algebra. Note that the maximum norm of a function in $A(\mathbb{T})$ is bounded by its $A(\mathbb{T})$-norm, and conversely, the $A(\mathbb{T})$-norm is bounded by the maximum norm of the function and its derivative, see \cite[Section I.6]{katznelson76}:
    \begin{equation}\label{eq:katz}
    \| f \|_{C(\mathbb{T})} \le \| f \|_{A(\mathbb{T})} \quad\text{ and }\quad
    \| f \|_{A(\mathbb{T})} \le c_1\,\| f \|_{C^1(\mathbb{T})}.
    \end{equation}
The space $C([0,T],A(\mathbb{T}))$ is the Banach space of $A(\mathbb{T})$-valued continuous functions on the interval $[0,T]$, with $\| d \|_{C([0,T],A(\mathbb{T}))}=\max_{0\le t \le T} \| d(t,\cdot) \|_{A(\mathbb{T})}$.

\begin{lemma} \label{lem:defect-wiener}
In the situation of Theorem~\ref{th:conv-1}, the defect \eqref{eq:defect} is bounded in the Wiener algebra norm by
$$
\| d \|_{C([0,T],A(\mathbb{T}))} \le c\eps(\tau^2 + h^2),
$$
where $c$ is independent of  $\eps$, $\tau$, $h$, and $n$ with $t_n=n\tau\le T$.
\end{lemma}

\begin{proof} We define
\begin{align*}
 &\wt d(x,t) := d(x,t) \, \e^{-\iu(\kappa x-\omega t)/\eps}
 \\   
 &= \iu\eps\,\frac{a(t+\tau,x)-a(t-\tau,x)}{2\tau} 
 \\ 
 &\quad + \tfrac12\eps^2
 \biggl( \frac{a(t,x+h)-2a(t,x)+a(t,x-h)}{h^2} + 
    2 \frac{\iu\kappa}\eps \,\frac{a(t,x+h)-a(t,x-h)}{2h} \biggr)
 \\
    &\quad -\ \eps\lambda|a(t,x)|^2 a(t,x)
\end{align*}
and note that 
$$
\| \wt d(t,\cdot) \|_{A(\mathbb{T})} = \| d(t,\cdot) \|_{A(\mathbb{T})}.
$$
For the temporal finite difference we have by Taylor expansion 
$$
\frac{a(t+\tau,x)-a(t-\tau,x)}{2\tau} = \partial_t a(t,x) + \tau^2 R_\tau^{(1)}(t,x)
$$
with the continuously differentiable remainder in integral form,
$$
R_\tau^{(1)}(t,x)=\int_{-1}^1 \tfrac12 (1-|\theta|)^2 \,\partial_t^3 a(t+\theta\tau,x)\, d\theta,
$$
and similarly for the spatial finite differences with $\mathcal{O}(h^2)$ remainder terms in integral form. In view of the partial differential equation \eqref{eq:a} for $a$, this yields
$$
\wt d(t,x) = \iu\eps \tau^2 R_\tau^{(1)}(t,x) + \tfrac12\eps^2 h^2 R_h^{(2)}(t,x) + 
\iu\kappa\eps h^2 R_h^{(1)}(t,x)
$$
with continuously differentiable remainder terms, which have partial derivatives bounded independently of $\eps$, $\tau$ and $h$. So we obtain, uniformly for $0\le t \le T$,
$$
\| d(t,\cdot) \|_{A(\mathbb{T})} = \| \wt d(t,\cdot) \|_{A(\mathbb{T})} \le 
c_1 \, \| \wt d(t,\cdot) \|_{C^1(\mathbb{T})} \le c \eps(\tau^2 + h^2),
$$
which is the desired bound.
\end{proof}




\section{Stability}\label{sec:stability}

\subsection{Linear stability analysis in the Wiener algebra}\label{subsec:linear}
In this subsection we give linear stability results for the weighted leapfrog and Crank--Nicolson schemes. We bound numerical solutions corresponding to the linear Schrödinger equation \eqref{eq:schr} (without the nonlinearity) in the Wiener algebra norm, using Fourier analysis.

We momentarily omit the nonlinearity and interpolate the weighted leapfrog scheme \eqref{eq:scheme} from discrete spatial points $x_j=jh$ to arbitrary $x\in \mathbb{T}$ by setting 
\begin{align}
 \label{eq:scheme-x}
&\iu\eps\,\frac{\e^{\iu\alpha}u^{n+1}(x)-\e^{-\iu\alpha}u^{n-1}(x)}{2\tau}
\\ 
\nonumber
&+\frac{\eps^2}{2}\, \frac{\e^{-\iu\beta}(1+\iu\beta)u^n(x+h)-2u^n(x)+\e^{\iu\beta}(1-\iu\beta)u^n(x-h)}{h^2} =0.
\end{align}
We clearly have $u^n(x_j)=u^n_j$ of \eqref{eq:scheme} for all $n\ge 2$ if this holds true for $n=0$ and $n=1$. In particular, we have
$\max_j |u^n_j|\le \max_{x\in\mathbb{T}} |u^n(x)| \le \| u^n \|_{A(\mathbb{T})}$.

\begin{lemma}[Linear stability of the weighted leapfrog method]\label{th:stability}
    Under condition \eqref{eq:stability},
    the weighted leapfrog algorithm~\eqref{eq:scheme-x} without the nonlinear term is stable: There exists a norm $\vvvert\cdot\vvvert$
    on $A(\mathbb{T}) \times A(\mathbb{T})$,
    equivalent to the norm $\|\cdot\|_{A(\mathbb{T}) \times A(\mathbb{T})}$ uniformly in $\eps,\tau,h$ subject to the stability condition \eqref{eq:stability}, such that  
\[
\vvvert U^{n}\vvvert = \vvvert U^{n-1}\vvvert, \qquad\text{where}\ \ 
U^{n}=\begin{pmatrix}
        u^{n+1}\\u^{n}
    \end{pmatrix}.
\]
\end{lemma}

\begin{proof}
    Let $\hat{u}^n=(\hat{u}^n_k)$ be the sequence of Fourier coefficients of $u^n$, i.e., 
    \[
    u^n(x)=\sum_{k=-\infty}^{\infty}\e^{\iu k x}\,\hat{u}^n_k.
    \]
    Substituting this into \eqref{eq:scheme-x} yields, for all $j$,
    \begin{align*}
    \sum_{k}\e^{\iu k x_j}\Biggl(&\iu\eps\,\frac{\e^{\iu\alpha}\hat{u}^{n+1}_k-\e^{-\iu\alpha}\hat{u}^{n-1}_k}{2\tau}
    +\eps^2\,\frac{\gamma_k}{h^2}\,\hat{u}_k^n\Biggr)=0,
    \end{align*}
    where 
    \begin{align*}
    \gamma_k&=(\cos(\beta)+\beta\sin(\beta))\cos(kh)+\left(\sin(\beta)-\beta\cos(\beta)\right)\sin(kh)-1
    \\
    &= \cos(\beta-kh) + \beta \sin(\beta-kh)-1,
    \end{align*}
    which is bounded by
    $$
    |\gamma_k| \le \gamma :=1+\max(|\beta|,1) \quad\text{for all }k.
    $$
    We then have
    \begin{align*}
    \iu\eps\,\frac{\e^{\iu\alpha}\hat{u}^{n+1}_k-\e^{-\iu\alpha}\hat{u}^{n-1}_k}{2\tau}
    +\eps^2\,\frac{\gamma_k}{h^2}\,\hat{u}_k^n=0,
    \end{align*}
    which is equivalent to the system
    \[
    \begin{pmatrix}
        \hat{u}^{n+1}_k\\
        \hat{u}^{n}_k
    \end{pmatrix}
     =G_k\begin{pmatrix}
         \hat{u}^{n}_k\\
         \hat{u}^{n-1}_k
     \end{pmatrix},
    \]
    where
    \begin{align}\label{eq:matrix}
          &G_k=\begin{pmatrix}
        2\iu\mu_k\e^{-\iu\alpha} & \e^{-2\iu\alpha}\\
        1 & 0
    \end{pmatrix} 
    \quad \text{with }\ \mu_k=\frac{\eps\tau}{h^2}\,\gamma_k.
    \end{align}
    Let $\lambda_k^{+}, \lambda_k^{-}$ be the two roots of the characteristic polynomial
    \[
    \rho_k(\zeta)=\zeta^2-2\iu\mu_k\e^{-\iu\alpha}\zeta-\e^{-2\iu\alpha},
    \]
    i.e., 
    \[
    \lambda_k^{\pm}=\left(\iu\mu_k\pm(1-\mu_k^2)^{1/2}\right)\e^{-\iu\alpha}.
    \]
    Condition \eqref{eq:stability} ensures that $|\mu_k|<1$ and thus $|\lambda_k^{\pm}|=1$.
    The vectors $(\lambda_k^+,1)^\top$ and $(\lambda_k^-,1)^\top$ are eigenvectors of $G_k$ with eigenvalue $\lambda_k^+$ and $\lambda_k^-$, respectively. This is because (similar for $\lambda_k^-$)
    \[
    \begin{pmatrix}
        2\iu\mu_k\e^{-\iu\alpha} & \e^{-2\iu\alpha}\\
        1 &0
    \end{pmatrix}\begin{pmatrix}
        \lambda_k^+\\1
    \end{pmatrix}=\begin{pmatrix}
        2\iu\mu_k\e^{-\iu\alpha}\lambda_k^++\e^{-2\iu\alpha}\\ \lambda_k^+
    \end{pmatrix}=\lambda_k^+\begin{pmatrix}
        \lambda_k^+\\1
    \end{pmatrix}.
    \]
    Therefore $G_k$ is diagonalizable,
    \[
    P_k^{-1}G_kP_k=\Lambda_k=\text{diag}\{\lambda_k^+,\lambda_k^-\},
  \]
    and $\Lambda_k$ is a unitary matrix. Using the transformation matrix $P_k$, we have, for any vector $y\in\mathbb{C}^2$,
     \[
    |P_k^{-1}G_ky|_{2}=|\Lambda_k P_k^{-1}y|_{2} 
    =|P_k^{-1}y|_{2}.
    \]
    Therefore,
    \begin{equation}\label{eq:conserved}
    \begin{aligned}
     \vvvert U^{n}\vvvert 
    :=&\sum_k\left|P^{-1}_k
    \begin{pmatrix}
        \hat{u}_k^{n+1}\\\hat{u}_k^{n}
    \end{pmatrix}
    \right|_2
    =\sum_k\left|P^{-1}_kG_k
    \begin{pmatrix}
        \hat{u}_k^{n}\\\hat{u}_k^{n-1}
    \end{pmatrix}
    \right|_2
    \\
    =&\sum_k\left|P^{-1}_k
    \begin{pmatrix}
        \hat{u}_k^{n}\\\hat{u}_k^{n-1}
    \end{pmatrix}
    \right|_2
    =\vvvert U^{n-1}\vvvert.
    \end{aligned}
    \end{equation}
    Finally, we show that 
\[
\|P\|_2:=\max_{k}\|P_k\|_2\leq C_1,\quad \|P^{-1}\|_2:=\max_{k}\|P_k^{-1}\|_2\leq C_2,
\]
which yields that the newly introduced norm $\vvvert\cdot\vvvert$ is equivalent to $\|\cdot\|_{A(\mathbb{T}) \times A(\mathbb{T})}$. 
    Since
    \[
    P_k^*P_k=
    \begin{pmatrix}
    2 &1+\overline{\lambda_k^+}\lambda_k^-\\
    1+\overline{\lambda_k^-}\lambda_k^+ &2
    \end{pmatrix},
    \]
the eigenvalues of $P_k^*P_k$ can be calculated as $2(1\pm\mu_k)$. Since $|\mu_k|\le \theta < 1$ by condition~\eqref{eq:stability}, we have for all $k$ that
\[
\begin{aligned}
\|P_k\|_2&
=\sqrt{\lambda_{\text{max}}(P_k^*P_k)}<2, 
\\
\|P_k^{-1}\|_2&
=1/\sqrt{\lambda_{\text{min}}(P_k^*P_k)}\leq 1/\sqrt{2(1-\theta)}, 
\end{aligned}
\]
so that 
$$\tfrac12 \,\| U \|_{A(\mathbb{T}) \times A(\mathbb{T})} \le \vvvert U \vvvert \le \frac1{\sqrt{2(1-\theta)}}\,\| U \|_{A(\mathbb{T}) \times A(\mathbb{T})}
$$
for all $U\in A(\mathbb{T}) \times A(\mathbb{T})$.
\end{proof}

We similarly extend the weighted Crank--Nicolson algorithm \eqref{eq:CN} to all $x\in\mathbb{T}$ and omit the nonlinearity.

\begin{lemma}[Linear stability of the weighted Crank--Nicolson method]\label{lem:stability-cn}
    The weighted Crank--Nicolson algorithm \eqref{eq:CN} without the nonlinear term is unconditionally stable with
    \[
    \|u^{n+1}\|_{A(\mathbb{T})}=\|u^{n-1}\|_{A(\mathbb{T})}.
    \]
\end{lemma}

\begin{proof}  
    Substituting the Fourier series of $u^n$
    into \eqref{eq:CN} without the nonlinear term yields
    \begin{align*}
    \sum_{k}\e^{\iu k x}\Biggl(\iu\eps\frac{\e^{\iu\alpha}\hat{u}^{n+1}_k-\e^{-\iu\alpha}\hat{u}^{n-1}_k}{2\tau}
    +\eps^2\,\frac{\gamma_k}{h^2}\frac{\e^{\iu\alpha}\hat{u}_k^{n+1}+\e^{-\iu\alpha}\hat{u}_k^{n-1}}{2}\Biggr)=0,
    \end{align*}
    again with $\gamma_k=\cos(\beta-kh) + \beta \sin(\beta-kh)-1$, which leads to
    \[
    \left(\frac{\iu\eps}{2\tau}+
        \eps^2\,\frac{\gamma_k}{2h^2}
    \right)\e^{\iu\alpha}\hat{u}_k^{n+1}
    =\left(\frac{\iu\eps}{2\tau}-
    \eps^2\,\frac{\gamma_k}{2h^2}
    \right)\e^{-\iu\alpha}\hat{u}_k^{n-1}.
    \]
    Therefore we have $|\hat{u}_k^{n+1}|=|\hat{u}_k^{n-1}|$ for all $k$, which yields  the result.
\end{proof}

\subsection{Nonlinear stability}\label{subsec:nonlinear}
\begin{lemma}[Nonlinear stability of the weighted leapfrog method] \label{lem:stability}
    Let the function $u\in C([0,T],A(\mathbb{T}))$ be arbitrary and let the corresponding defect $d$ be defined by \eqref{eq:defect}.
    Under condition \eqref{eq:stability}, the interpolated numerical solution of \eqref{eq:scheme}, interpolated to all $x\in\mathbb{T}$ as in \eqref{eq:scheme-x} (but now with the nonlinear term included), satisfies the bound, for $t_n=n\tau\le T$
    \[
    \|u^n-u(t_n,\cdot)\|_{A(\mathbb{T})}\le C \Bigl(\|u^0-u(0,\cdot)\|_{A(\mathbb{T})}+\|u^1-u(t_1,\cdot)\|_{A(\mathbb{T})}
    + \eps^{-1}\| d \|_{C([0,T],A(\mathbb{T}))} \Bigr),
    \]
    where $C$ is independent of $\eps$, $\tau$, $h$, and $n$ with $t_n\le T$, but depends on $T$ and on upper bounds of the above term in big brackets and of the $C([0,T],A(\mathbb{T}))$ norm of~$u$.
\end{lemma}
\begin{proof}
We define the error function $e^n(x)=u^n(x)-u(t_n,x)$, which satisfies
\[
\begin{aligned}
&\e^{\iu\alpha}e^{n+1}(x)-\e^{-\iu\alpha}e^{n-1}(x)\\
=&\,\frac{\iu\eps\tau}{h^2}\left(\e^{-\iu\beta}(1+\iu\beta)e^{n}(x+h)-2e^n(x)+\e^{\iu\beta}(1-\iu\beta)e^{n}(x-h)\right)\\
&-2\iu\lambda\tau \left(|u^n(x)|^2u^n(x)-|u(t_n,x)|^2u(t_n,x)\right)
-2\iu\tau\,\eps^{-1}\,d(t_n,x).
\end{aligned}
\]
The Fourier coefficient of $e^n$ then satisfies 
\[
\begin{aligned}
&\e^{\iu\alpha}\hat{e}_k^{n+1}-\e^{-\iu\alpha}\hat{e}^{n-1}_k\\
=&\,\frac{2\iu\eps\tau}{h^2}\,\frac{\cos(\beta)\cos(kh)+\left(\sin(\beta)-\beta\right)\sin(kh)-1}{h^2}\hat{e}^n_k
\\
&-2\iu\lambda\tau \mathcal{F}\left(|u^n(x)|^2u^n(x)-|u(t_n,x)|^2u(t_n,x)\right)(k)-2\iu\tau\eps^{-1}\hat{d}^n_k.
\end{aligned}
\]
This equation can be written in the one-step form
\begin{align*}
\begin{pmatrix}
\hat{e}^{n+1}_k\\ \hat{e}^{n}_k
\end{pmatrix} =
G_k
\begin{pmatrix}
\hat{e}^{n}_k\\ \hat{e}^{n-1}_k
\end{pmatrix}
&-2\iu\lambda\tau\,\e^{-\iu\alpha}
\begin{pmatrix}
\mathcal{F}\left(|u^n(x)|^2u^n(x)-|u(t_n,x)|^2u(t_n,x)\right)(k)\\ 0
\end{pmatrix}
\\
& -2\iu\tau \eps^{-1}\e^{-\iu\alpha}
\begin{pmatrix}
\hat{d}^n_k\\ 0
\end{pmatrix},
\end{align*}
where $G_k$ is defined in \eqref{eq:matrix}. 
 
We define the error vector as ${\mathcal{E}}^{n}=\begin{pmatrix}        e^{n+1}\\
        e^{n}
     \end{pmatrix}$.
Multiplying the above equation by $P_k^{-1}$ and summing over $k$ gives
\[
\begin{aligned}
 &\vvvert \mathcal{E}^n\vvvert=\,\sum_k\left|P^{-1}_k
    \begin{pmatrix}
        \hat{u}_k^{n+1}\\\hat{u}_k^{n}
    \end{pmatrix}
    \right|_2\\
    \leq&\,\sum_k\left|P^{-1}_kG_k
\begin{pmatrix}
\hat{e}^{n}_k\\ \hat{e}^{n-1}_k
\end{pmatrix}\right|_2+c_0\tau\sum_k\left|P^{-1}_k\begin{pmatrix}
\mathcal{F}\left(|u^n(x)|^2u^n(x)-|u(t_n,x)|^2u(t_n,x)\right)(k)\\ 0
\end{pmatrix}
    \right|_2\\
&+\tilde{c}_0\tau\eps^{-1}\sum_k\left|P^{-1}_k\begin{pmatrix}
\hat{d}^n_k\\ 0
\end{pmatrix}\right|_2.
 \end{aligned}
 \]
 By \eqref{eq:conserved}, the first term on the right-hand side is $\vvvert{\mathcal{E}}^{n-1}\vvvert$. The second term can be estimated as follows (we use $\lesssim$ to denote $\leq C$ for some constant $C$):
\[
\begin{aligned}
&\sum_k\Biggl|P^{-1}_k
\binom{\mathcal{F}\!\left(|u^n(x)|^2u^n(x)-|u(t_n,x)|^2u(t_n,x)\right)(k)}{0}\Biggr|_2\\
\lesssim &\,
\left\|\,|u^n(x)|^2u^n(x)-|u(t_n,x)|^2u(t_n,x)\,\right\|_{A(\mathbb T)}\\
=&\,\left\|
\left(|u^n(x)|^2+|u(t_n,x)|^2\right)\left(u^n(x)-u(t_n,x)\right)
+ u^n(x)\left(\bar{u}^n(x)-\bar{u}(t_n,x)\right)u(t_n,x)
\,\right\|_{A(\mathbb T)}\\
\lesssim &\,
\|u^n-u(t_n,\cdot)\|_{A}
 \;\lesssim\; \vvvert{\mathcal{E}}^{\,n-1}\vvvert,
\end{aligned}
\]
 where we have used the estimate \eqref{eq:tri} for the nonlinear term and the norm equivalence between $\vvvert\cdot\vvvert$ and $\|\cdot\|_{A(\mathbb{T})\times A(\mathbb{T})}$ as stated in Lemma~\ref{th:stability}. We then have
\[
\begin{aligned}
 \vvvert \mathcal{E}^n\vvvert\leq&\,
(1+c\tau)
\vvvert{\mathcal{E}}^{n-1}\vvvert
+\widetilde c\tau \eps^{-1} \|d(t_n,\cdot)\|_{A(\mathbb{T})}\\
\leq& \, (1+c\tau)^n
\vvvert{\mathcal{E}}^{0}\vvvert
+\widetilde c\tau \eps^{-1}\sum_{j=1}^{n}(1+c\tau)^{n-j}\|d(t_j,\cdot)\|_{A(\mathbb{T})}
\\
\leq& \,\e^{cn\tau}\vvvert\mathcal{E}^{0}\vvvert
+\widetilde c\tau \eps^{-1}\frac{\e^{cn\tau}-1}{c\tau}\sup_{t\in[0,T]}\|d(t,\cdot)\|_{A(\mathbb{T})},
\end{aligned}
\]
which yields the result, using the norm equivalence once more.
\end{proof}

An analogous result holds true for the weighted Crank--Nicolson method, with essentially the same proof, now based on Lemma~\ref{lem:stability-cn}.

The proof of Theorem \ref{th:conv-1} is then finished by combining Lemmas~\ref{lem:defect-wiener} and \ref{lem:stability}.

\bigskip\bigskip
\centerline{\bf Part II. Multiphase initial conditions}
\bigskip
For the multiphase problem we apply the weighted finite difference methods multiple times, corresponding to each of the different wave vectors $\kappa_m$ and associated frequencies $\omega_m$. The nonlinearity needs to be treated in a special way.

\section{Two opposite phases}\label{sec:two}
As an illustration of the procedure for multiphase initial data, we consider in this section the particular case of two initial wave packets having opposite wave numbers.

\subsection{Modulated Fourier expansion}
The following result provides an $\mathcal{O}(\eps^2)$ approximation to the solution. It is a special case of Theorem~\ref{th:MFE} given in the next section for the general multiphase case.   

\begin{proposition}[Modulated Fourier expansion for two opposite initial phases]\label{prop:MFE-2}
Let $u(t,x)$ be the solution to \eqref{eq:schr} with initial data given by
\[
u(0,x) = a^0_{1}(x)\, \mathrm{e}^{\iu \kappa x / \eps} + a^0_{-1}(x)\, \mathrm{e}^{-\iu \kappa x / \eps}
\]
with $\kappa\ne 0$ and with smooth functions $a^0_{\pm1}$.  Then, $u(t,x)$ admits a modulated Fourier expansion
\[
u(t,x) = u_{\text{\rm MFE}}(t,x) + e(t,x),
\]
where, with $\kappa_{\pm 1}= \pm \kappa$ and $\omega_{\pm 1}=\tfrac12\kappa^2$, 
and with $\kappa_{\pm3}=\pm 3\kappa$ and $\omega^\star_{\pm3}=\tfrac12(\pm 3\kappa)^2$,  
$\omega_{\pm3}=\tfrac12\kappa^2\neq\omega^\star_{\pm3}$ 
and with smooth modulation functions $a_{\pm1}$, $b^\star_{\pm3}$ and $b_{\pm3}$ defined below,
\[
\begin{aligned}
u_{\text{\rm MFE}}(t,x)=&\sum_{r =\pm1} a_r(t,x)\, \mathrm{e}^{\iu (\kappa_r x - \omega_r t)/\eps}\\
+&\eps\lambda \sum_{\nu=\pm3} b^\star_\nu(t,x)\, \mathrm{e}^{\iu (\kappa_\nu x - \omega_\nu^\star t)/\eps}+\eps\lambda\sum_{\nu=\pm3}b_\nu(t,x)\, \mathrm{e}^{\iu (\kappa_\nu x - \omega_\nu t)/\eps}.
\end{aligned}
\] 
The function $a_{1}(t,x)$ is the solution of an advected nonlinear Schr\"odinger equation with smooth initial data,
\begin{equation}\label{eq:a1}
\begin{aligned}
&\partial_t a_{1} + \kappa_1 \partial_x a_{1} - \frac{\iu\eps}{2} \partial_x^2 a_{1} = 
-\iu\lambda\left(\left(|a_{1}|^2 + 2|a_{-1}|^2\right)a_{1}+\eps\lambda(2 a_{-1}\bar{a}_{1}b_{3}+ a_{-1}\bar{b}_{-3}a_{-1})\right), 
\\
&a_{1}(0,x)=a^0_{1}(x),
\end{aligned}
\end{equation}
and $a_{-1}(t,x)$ satisfies the same equation where all subscripts have reversed signs. 

The function $b_{3}(t,x)$ (and analogously $b_{-3}$ with opposite subscripts) is given by the formula, with $\delta_3=\omega_3 - \omega_3^\star=-4\kappa^2\ne 0$,
\[
b_{3}(t,x) = \frac{1}{\delta_3}( a_{1} \bar{a}_{-1} a_{1})(t,x),
\]
which turns \eqref{eq:a1} into a quintic advected Schr\"odinger equation.

The function $b^\star_{3}(t,x)$ (and analogously $b_{-3}^\star$ with opposite subscripts)
solves the advected linear Schr\"odinger equation
$$
\partial_t b^\star_{3} +\kappa_3\partial_xb^\star_{3}
- \frac{\iu\eps}{2} \partial_x^2 b_{3}^\star = -2\iu\lambda (|a_{1}|^2+|a_{-1}|^2)b_{3}^\star, \quad\ b^\star_{3}(0,x)=-b_{3}(0,x).
$$
With these coefficient functions $a_{\pm1}$, $b^\star_{\pm3}$ and $b_{\pm3}$, the error $e(t,x)$ is bounded in the maximum norm by
\[
\|e\|_{C([0,T]\times\mathbb{T)}}\leq C\eps^2.
\] 
\end{proposition}

This result is basic for constructing a numerical method with an $O(\tau^2+h^2+\eps^2)$ error bound without restrictions on the ratios $\tau/\eps$ and $h/\eps$.

\subsection{Weighted finite difference methods} \label{subsec:wfdm}
We extend the weighted leapfrog and Crank--Nicolson methods of Part I to the case of two initial phases. To simplify the notation, we define the weighted leapfrog finite difference operator for $u$, which depends on the parameters $\alpha=\tfrac12\kappa^2\tau/\eps$ and $\beta=\kappa h/\eps$:
\begin{equation}\label{DLF}
\mathcal{D}_{\text{LF}}^{\alpha,\beta} u\big|^n_j :=
\iu\eps\,\frac{\e^{\iu\alpha}u_j^{n+1}-\e^{-\iu\alpha}u_j^{n-1}}{2\tau}
+\frac{\eps^2}{2}\, \frac{\e^{-\iu\beta}(1+\iu\beta)u^n_{j+1}-2u^n_j+\e^{\iu\beta}(1-\iu\beta)u^n_{j-1}}{h^2}.
\end{equation}
We consider four schemes that are natural extensions of one another and guide the derivation of the final scheme. In the first cases we approximate the solution by
$$
u(t_n,x_j)\approx u^n_j = \sum_{r =\pm1} u_{[r],j}^n
$$
where $u_{[r],j}^n$ is to approximate $u_r(t,x)=a_r(t,x)\, \mathrm{e}^{\iu (\kappa_r x - \omega_r t)/\eps}$.
In an $\eps$-asymptotically more accurate scheme, we then make a refined approximation
$$
u(t_n,x_j)\approx u^n_j = \sum_{r =\pm1} u_{[r],j}^n
+\sum_{\nu=\pm3} w_{[\nu],j}^{\star,n}+\sum_{\nu =\pm3} w_{[\nu],j}^n,
$$
where
$w_{[\nu],j}^{\star,n}$ is to approximate $w^\star_\nu(t,x)=\eps\lambda\, b^\star_\nu(t,x)\, \mathrm{e}^{\iu (\kappa_\nu x - \omega^\star_\nu t)/\eps}$ and $w_{[\nu],j}^{n}$ is to approximate $w_\nu(t,x)=\eps\lambda\, b_\nu(t,x)\, \mathrm{e}^{\iu (\kappa_\nu x - \omega_\nu t)/\eps}$  at $t=t_n$ and $x=x_j$.

\medskip
\noindent{\bf Case 0 (Naive coupling).}  
We directly extend the single-mode scheme to two modes with fully coupled nonlinearity
\[
\mathcal{D}_{\mathrm{LF}}^{\alpha,\beta}u_{[1]}\big|^n_j = \eps\lambda\big|u_{[1],j}^n + u_{[-1],j}^n\big|^2\,u_{[1],j}^n,
\]
and apply the same scheme to $u_{[-1]}$ with all subscripts of opposite signs  and $\beta$ replaced by $-\beta$. While straightforward, this leads to incorrect numerical results except for very small $\tau\ll\eps$ and $h\ll\eps$.

\medskip
\noindent{\bf Case 1 (Separated phases).}  
To better match the system for $a_{\pm1}$, we separate the nonlinear interactions as
\[
\mathcal{D}_{\mathrm{LF}}^{\alpha,\beta}u_{[1]}\big|^n_j = \eps\lambda\big(|u_{[1],j}^n|^2 + 2|u_{[-1],j}^n|^2 + u_{[1],j}^n\bar{u}_{[-1],j}^n \big)u_{[1],j}^n,
\]
and apply the same scheme to $u_{[-1]}$ with all subscripts of opposite signs and $\beta$ replaced by $-\beta$.  The mixed term  $ u_{[1],j}^n \bar{u}_{[-1],j}^n u_{[1],j}^n$ is oscillatory with wave number $3\kappa$ and nonresonant frequency $\omega=\tfrac 12 \kappa^2\ne \tfrac12(3\kappa)^2$. It contributes an  $\mathcal{O}(\tau)$ error if $\tau>\eps$ and $\alpha=\omega\tau/\eps$ is bounded away from multiples of $2\pi$. This can be traced back to partial summation and the geometric sum formula
\[
\tau\sum_{k=0}^{n}\e^{-\iu k \alpha}=\tau\,\frac{1-\e^{-\iu (n+1)\alpha}}{1-\e^{-\iu\alpha}}.
\]
The error is then bounded by the minimum of $\mathcal{O}(\tau + h^2 + \eps)$ and $\mathcal{O}((\tau^2 + h^2)/\eps^3)$, where the latter term results from a standard error analysis using Taylor expansion of the solution $u$ (of interest when $\eps$ is not very small and $h\le\eps^2$).

\medskip
\noindent{\bf Case 2 ($\eps$-asymptotically first order accurate).}  
To improve the order of accuracy in the time step $\tau$ for small $\eps$, we remove the high-frequency oscillations and set
\[
 \mathcal{D}_{\mathrm{LF}}^{\alpha,\beta}u_{[1]}\big|^n_j = \eps\lambda\big(|u_{[1],j}^n|^2 + 2|u_{[-1],j}^n|^2\big)u_{[1],j}^n.\\[1mm]
\]
The formula for $u_{[-1]}$ is obtained by reversing the signs of all subscripts and replacing $\beta$ by $-\beta$. 
This scheme achieves accuracy of order $\mathcal{O}(\tau^2 + h^2 + \eps)$ when $\eps$ is small, but deteriorates as $\eps$ increases.

\medskip
\noindent{\bf Case 3 ($\eps$-asymptotically second order accurate).}  
To improve the accuracy for small $\eps$, we add the $\mathcal{O}(\eps)$ terms and assign the high-frequency oscillations to higher modes:
\[
\begin{aligned}
\mathcal{D}_{\mathrm{LF}}^{\alpha,\beta}u_{[1]}\big|^n_j &= \eps\lambda\big((|u_{[1],j}^n|^2 + 2|u_{[-1],j}^n|^2)u_{[1],j}^n\\
&\quad\qquad+2u_{[-1],j}^n\bar{u}_{[1],j}^n w_{[3],j}^n+u_{[-1],j}^n\bar{w}_{[-3],j}^nu_{[-1],j}^n\bigr)\\[2mm]
&\text{with }\ \,w^{n+1}_{[3],j} = \frac{\eps\lambda}{\delta_3}\, u_{[1],j}^{n+1} \bar{u}_{[-1],j}^{n+1} 
 u_{[1],j}^{n+1},
\\[2mm]
\mathcal{D}_{\mathrm{LF}}^{9\alpha,3\beta}w^\star_{[3]}\big|^n_j &=2\eps\lambda(|u_{[1],j}^n|^2+|u_{[-1],j}^n|^2)w_{[3],j}^{\star,n}.
\end{aligned}
\]
The formulas for the components with negative subscripts are obtained by reversing the signs of all subscripts and replacing  $\beta$ by $-\beta$. This scheme achieves accuracy of order $\mathcal{O}(\tau^2 + h^2 + \eps^2)$.

\medskip
\noindent{\bf Extended leapfrog algorithm.} To ensure uniform accuracy independently of $\eps$, we filter the oscillatory terms using a switching function:
\begin{equation}\label{eq:scheme1-two}
\begin{aligned}
\mathcal{D}_{\mathrm{LF}}^{\alpha,\beta}u_{[1]}\big|^n_j &= \eps\lambda\big((|u_{[1],j}^n|^2 + 2|u_{[-1],j}^n|^2+ \chi\,u_{[1],j}^n\bar{u}_{[-1],j}^n )u_{[1],j}^n\\
&\quad\qquad+2u_{[-1],j}^n\bar{u}_{[1],j}^n w_{[3],j}^n+u_{[-1],j}^n\bar{w}_{[-3],j}^nu_{[-1],j}^n\bigr),\\[1mm]
\mathcal{D}_{\mathrm{LF}}^{9\alpha,3\beta}w^\star_{[3]}\big|^n_j &=2\eps\lambda(|u_{[1],j}^n|^2+|u_{[-1],j}^n|^2)w_{[3],j}^{\star,n},\\[1mm]
\delta_3\,w^{n+1}_{[3],j} &= (1-\chi)\eps\lambda u_{[1],j}^{n+1} \bar{u}_{[-1],j}^{n+1} u_{[1],j}^{n+1},
\end{aligned}
\end{equation}
where $\chi = 1$ if ${h^2\leq c\,\eps^5}$ and zero otherwise. The initial conditions are given by
\[
\begin{aligned}
    u_{[1]}(0,x)&=a^0_{1}(x)\, \mathrm{e}^{\iu \kappa x / \eps},\\
    w_{[3]}(0,x)&=(1-\chi)\frac{\eps\lambda}{\delta_3}u_{[1]}\bar{u}_{[-1]}u_{[1]}(0,x)\\
    w^\star_{[3]}(0,x)&=-w_{[3]}(0,x).
\end{aligned}
\]
The formulas for the components with negative subscripts are obtained by reversing the signs of all subscripts and replacing $\beta$ by $-\beta$.

\medskip
\noindent{\bf Extended Crank--Nicolson algorithm.}  
We define the weighted finite difference operator
\begin{equation} \label{DCN}
\mathcal{D}_{\mathrm{CN}}^{\alpha,\beta}u\big|^n_j :=
\iu\eps\,\frac{\e^{\iu\alpha}u_j^{n+1} - \e^{-\iu\alpha}u_j^{n-1}}{2\tau}
+ \frac{\eps^2}{2}\, \frac{\e^{-\iu\beta}(1+\iu\beta)\tilde{u}^n_{j+1} - 2\tilde{u}^n_j + \e^{\iu\beta}(1-\iu\beta)\tilde{u}^n_{j-1}}{h^2}, 
\end{equation}
where $\tilde{u}_j^n = (\e^{\iu\alpha}u_j^{n+1} + \e^{-\iu\alpha}u_j^{n-1})/{2}$. Following the same strategy, we construct a Crank--Nicolson-type discretization
\begin{equation}\label{eq:scheme2-two}
\begin{aligned}
\mathcal{D}_{\mathrm{CN}}^{\alpha,\beta}u_{[1]}\big|^n_j &= \eps\lambda\Bigl(\bigl((|u_{[1],j}^{n+1}|^2 + |u_{[-1],j}^{n-1}|^2)/2 + |u_{[-1],j}^{n+1}|^2 + |u_{[-1],j}^{n-1}|^2
\\[1mm]
&\qquad \quad + \chi\, {\tilde{u}}_{[1],j}^n \bar{\tilde{u}}_{[-1],j}^n\bigr)\tilde{u}_{[1],j}^n  \\
&\qquad \quad +2\tilde{u}_{[-1],j}^n\bar{\tilde{u}}_{[1],j}^n \tilde{w}_{[3],j}^n+\tilde{u}_{[-1],j}^n\bar{\tilde{w}}_{[-3],j}^n\tilde{u}_{[-1],j}^n\Bigr) 
\\[2mm]
\text{with }\ 
w^{n+1}_{[3],j} &= (1-\chi)\,\frac{\eps\lambda}{\delta_3}\, u_{[1],j}^{n+1} \bar{u}_{[-1],j}^{n+1} u_{[1],j}^{n+1},
\\[2mm]
\mathcal{D}_{\mathrm{CN}}^{9\alpha,3\beta}w^\star_{[3]}\big|^n_j &=\eps\lambda(|u_{[1],j}^{n-1}|^2+|u_{[1],j}^{n+1}|^2+|u_{[1],j}^{n-1}|^2+|u_{[-1],j}^{n+1}|^2)\tilde{w}_{[3],j}^{\star,n}.
\end{aligned}
\end{equation}
The initial conditions are the same as those specified for the extended leapfrog scheme.

The following result is a special case of Theorem~\ref{th:conv-m} given below for the general multiphase case.

\begin{proposition}[$\eps$-uniform convergence in the maximum norm]
Under the assumptions of Proposition~\ref{prop:MFE-2}, and in the case of the leapfrog scheme subject to the stability condition $\eps\tau< \min(h^2/\gamma(\beta),h^2/\gamma(3\beta))$ 
with $\gamma$  defined by \eqref{eq:stability}, the following error estimate holds for both methods \eqref{eq:scheme1-two} and \eqref{eq:scheme2-two}:
\[
|u(t_n, x_j) - u^n_j| \leq \min\left(C_0(\tau^2 + h^2 + \eps^2),\; C_1\frac{\tau^2 + h^2}{\eps^3}\right) \leq C(\tau^{4/5} + h^{4/5}),
\]
uniformly for $t_n = n\tau \leq T$ and $x_j = jh$. The constants $C_0,C_1$ and $C$ are independent of $\eps\in(0,1]$, the time step $\tau$, and the mesh size $h$, and independent of $j$ and $n$ for $t_n \leq T$.
\end{proposition}

\subsection{Numerical experiments.}
In this numerical test, we consider the one-dimensional semiclassical nonlinear Schr\"odinger equation
\[
\iu\eps\partial_t u+\tfrac12{\eps^2}\partial_{xx}u= \eps |u|^2u
\]
with the initial condition
\[
 u(0,x)=\tfrac{1}{2}\e^{-x^2}\e^{\iu x/\eps}+\tfrac{1}{2}\e^{-x^2}\e^{-\iu x/\eps}.
\]
The spatial domain and final time $T$ are the same as in the previous experiments. The solution error is again evaluated using the discrete $L^\infty$ norm. The switching function is chosen as $\chi = 1$ if ${h^2\leq 5\,\eps^5}$ and zero otherwise.

\begin{figure}[h]
\centerline{
\includegraphics[scale=0.5]{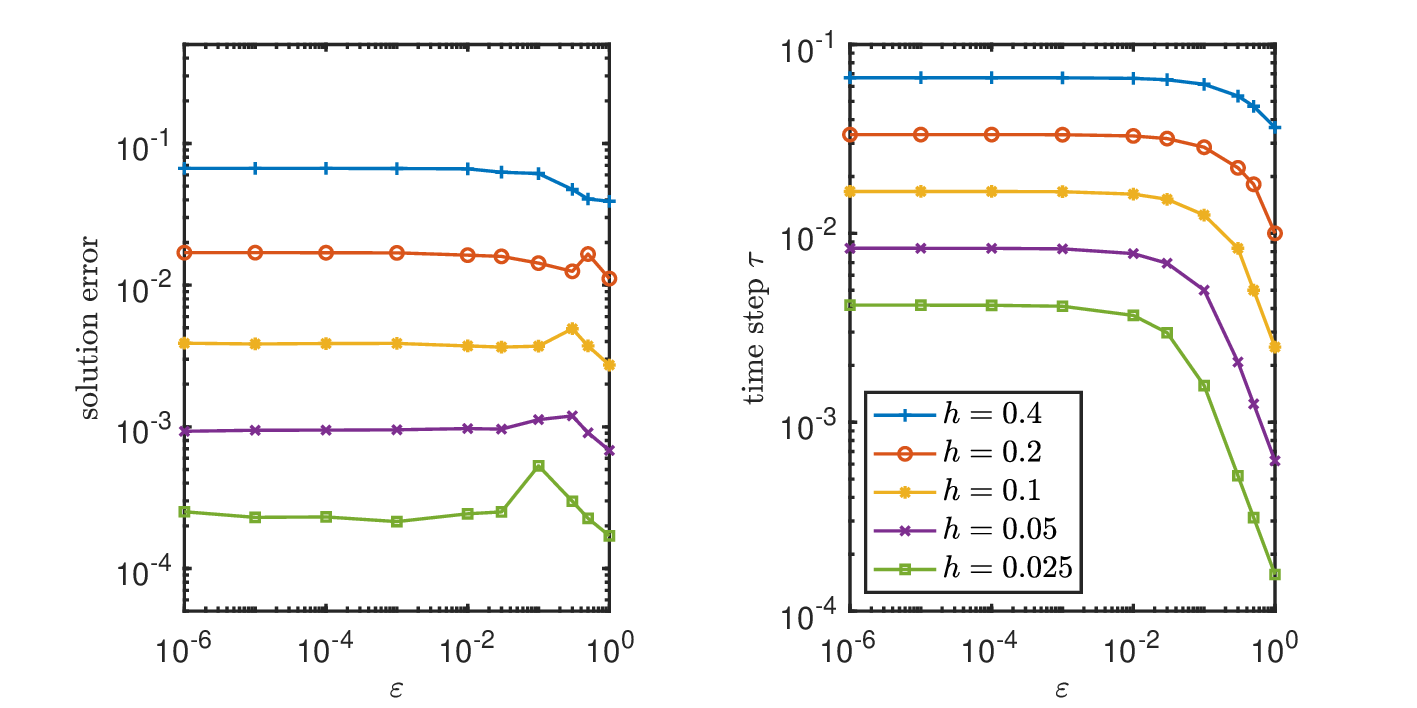}}
\centerline{
\includegraphics[scale=0.5]{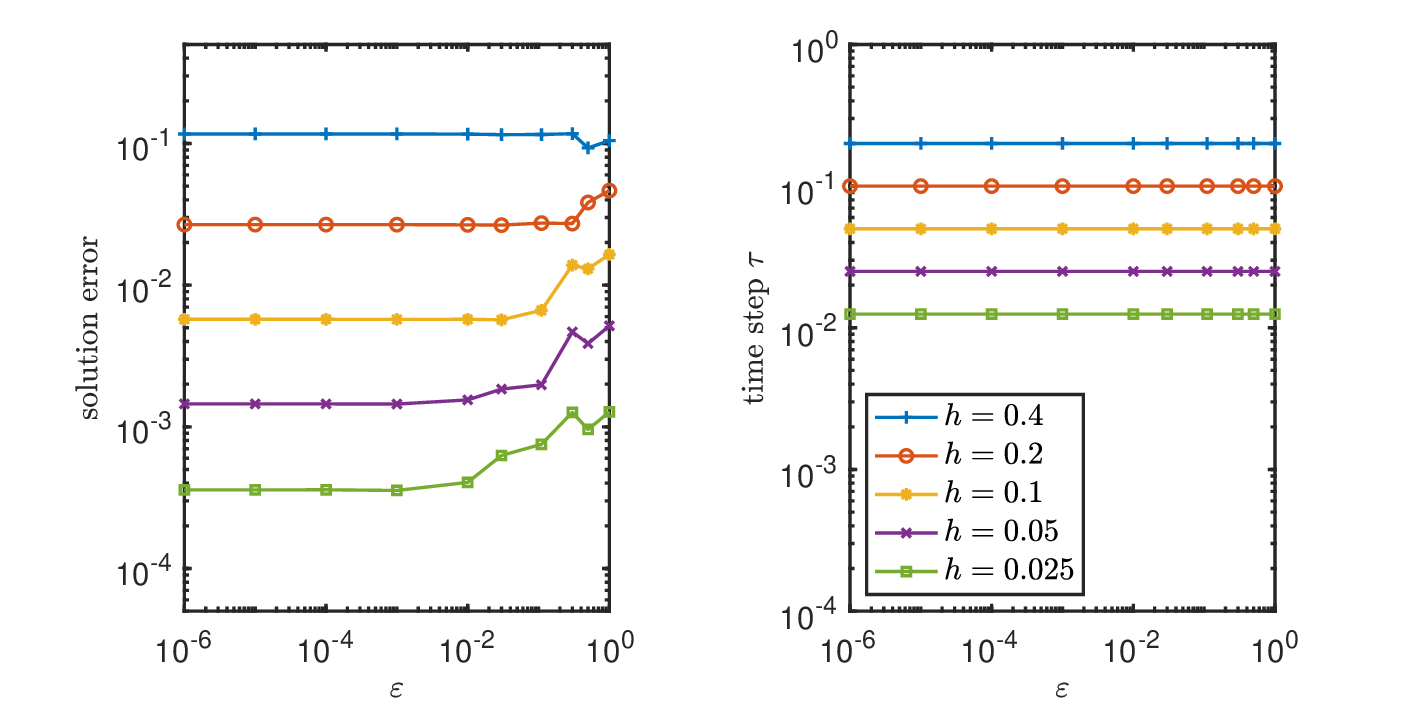}
}
\caption{Error and time stepsize vs. $\eps$ with different $h$ for the weighted leapfrog method (top row) and weighted Crank--Nicolson method (bottom row).}
\label{fig:two}
\end{figure}

Figure \ref{fig:two} displays the absolute error in $u$ plotted against $\eps$ for several fixed values of $h$, using the weighted leapfrog (top row) and the weighted Crank--Nicolson method (bottom row). We chose $\tau=\min(h/2,h^2/(2\eps\gamma_3))$ with $\gamma_3=\gamma(3\beta)$ for the weighted leapfrog method while $\tau=h/2$ for the weighted Crank--Nicolson method. The corresponding time step $\tau$ as a function of $\eps$ is shown in the right panel. It is observed that the error levels off at a constant value proportional to $h^2$ for 
small $\eps$. 

\section{General multiphase problem: modulated Fourier expansion of the exact solution}\label{sec:multi}
We now consider the semiclassical cubic Schr\"odinger equation \eqref{eq:schr} with multiphase initial data given by~\eqref{eq:init} with smooth profile functions $a_m^0:\mathbb{T}^d\to\mathbb{C}$ for $m=1,\dots,M$, and with pairwise different wave vectors $\kappa_m\in\mathbb{R}^d$ 
and with associated frequencies $\omega_m=\tfrac12 |\kappa_m|^2$, where $|\cdot|$ is the Euclidean norm.

For a multi-index $\mu=(m_0,\dots,m_{2l}) \in \{1,\dots,M\}^{2l+1}$ with  $l\ge 0$ we denote
$$
\kappa_\mu= \sum_{i=0}^{2l} (-1)^i \kappa_{m_i} 
\quad \text{ and } \quad
\omega_\mu= \sum_{i=0}^{2l} (-1)^i \omega_{m_i} .
$$
We call the multi-index $\mu$ {\it resonant} if $\omega_\mu=\tfrac12 |\kappa_\mu|^2$.

We construct a sequence of wave vectors that correspond to resonant multi-indices.
Given a set $K=\{\kappa_1,\dots,\kappa_R\}$ (with $R\ge M$) of wave vectors with associated frequences $\omega_r=\tfrac12|\kappa_r|^2$ for $r=1,\dots,R$, we augment this set as follows:
Let $\mu=(i,j,k)$ be a multi-index in $\{1,\dots,R\}^3$, and let $p,q\in \{1,\dots,R\}$. 

\:(i) If $\mu$ is resonant and $\kappa_\mu\notin K$, then we add $\kappa_\mu$ to $K$.

(ii) If $\mu$ is nonresonant and $(\mu,p,q)$ is resonant and $\kappa_{(\mu,p,q)}\notin K$, then we add $\kappa_{(\mu,p,q)}$ to $K$. Similarly, if $\mu$ is nonresonant and $(p,\mu,q)$ is resonant and $\kappa_{(p,\mu,q)}\notin K$, then we add $\kappa_{(p,\mu,q)}$ to $K$.

By these two rules, we augment $K$ to a set $\wh K$ with $\wh R \ge R$ elements. By construction, $K\subseteq\wh K$. We enumerate the elements of $\wh K$,
$$
\wh K = \{\kappa_1,...\kappa_R,\kappa_{R+1},\dots, \kappa_{\wh R} \}
$$
and for the corresponding frequencies we have $\omega_r=\tfrac12|\kappa_r|^2$ for $r=1,\dots,\wh R$.

We iterate the map $K\mapsto \wh K$: Starting from $K_0=\{\kappa_1,\dots,\kappa_M\}$, we define the set $K_{k+1}=\wh K_k$ for $k=0,1,2,\dots$
We make the following assumption.

\medskip\noindent
{\bf Assumption 1.} (Saturation condition) {\it There exists an integer $k_{\star}\ge 0$ such that 
    ${K}_{k_{\star}+1} = {K}_{k_{\star}}$.
    }

\medskip\noindent
In the following we write $K={K}_{k_{\star}}$ and enumerate the elements of this set, which contains $\kappa_1,\dots,\kappa_M$:
$$
K=\{ \kappa_1,\dots,\kappa_M,\kappa_{M+1},\dots,\kappa_R \}
$$
For $r=1,\dots,R$, the frequency associated with 
$\kappa_r$ is then
$\omega_r=\tfrac12|\kappa_r|^2$.  

\medskip\noindent
{\it Remark.} (a) In the case of dimension $d=1$, we have generically $k_\star=0$. Rule (i) cannot add new resonant wave numbers, as is seen from the basic formula
\begin{equation}\label{res-formula}
\tfrac12(\kappa_i-\kappa_j+\kappa_k)^2 - (\tfrac12\kappa_i^2-\tfrac12\kappa_j^2+\tfrac12\kappa_k^2) = (\kappa_i-\kappa_j)(\kappa_k-\kappa_j).
\end{equation}
Rule (ii) can add new resonant wave numbers only in exceptional cases. Using \eqref{res-formula} twice, we found that this can happen only for a two-parameter family of quintuples $(\kappa_i,\kappa_j,\kappa_k,\kappa_p,\kappa_q)$.

(b) The situation is much more intricate in higher dimensions $d>1$; see the discussion in \cite{carles2010multiphase}. Note that the higher-dimensional version of
\eqref{res-formula},
$$
\tfrac12|\kappa_i-\kappa_j+\kappa_k|^2 
- (\tfrac12|\kappa_i|^2-\tfrac12|\kappa_j|^2+\tfrac12|\kappa_k|^2) 
= (\kappa_i-\kappa_j)\cdot (\kappa_k-\kappa_j)
$$
with the Euclidean inner product $\cdot$ on the right-hand side, does no longer imply that
$\kappa_i=\kappa_j$ or $\kappa_k=\kappa_j$ when the left-hand side vanishes. 
However, $k_\star=0$ still appears to be generic.

\medskip\noindent
We let $\mathcal{N}$ be the set of all multi-indices $\nu=(i,j,k)$ with $i,j,k\in\{1,\dots,R\}$ that are nonresonant, i.e., with $\kappa_\nu=\kappa_i-\kappa_j+\kappa_k$ and $\omega_\nu=\omega_i-\omega_j+\omega_k$ we have $\omega_\nu \ne \tfrac12 |\kappa_\nu|^2$.
For $\nu\in \mathcal{N}$, we let $\omega_\nu^\star= \tfrac12 |\kappa_\nu|^2$ and
$\delta_\nu= \omega_\nu-\omega_\nu^\star\ne 0$.

\medskip\noindent
{\bf Assumption 2.} (Nonresonance condition) {\it For all $\nu\in \mathcal{N}$ and all $p,q,r=1,\dots,R$ with $q\ne r$,}
$$
    \omega_\nu^\star - \omega_q + \omega_r \ne \tfrac12 |\kappa_\nu - \kappa_q + \kappa_r |^2 
    \quad\text{and}\quad
    \omega_p - \omega_\nu^\star + \omega_r \ne \tfrac12 |\kappa_p - \kappa_\nu  + \kappa_r |^2 .
$$

In dimension $d=1$, it follows from \eqref{res-formula} that this condition is always satisfied. In higher dimensions, the condition still appears to be generic; 
cf.~\cite{carles2010multiphase}.

The following theorem provides an $\mathcal{O}(\eps^2)$ approximation to the solution of \eqref{eq:schr} with multiphase initial data \eqref{eq:init}.


\begin{theorem}[Modulated Fourier expansion for multiphase initial data]
\label{th:MFE}
Let $u(t,x)$ be the solution to~\eqref{eq:schr} with initial data given by~\eqref{eq:init}, with wave vectors $\kappa_m$ for which Assumptions 1 and 2 are fulfilled. Then, $u(t,x)$ admits a modulated Fourier expansion
\[
u(t,x) = u_{\mathrm{MFE}}(t,x)+ e(t,x),
\]
where, with the notation introduced above and with modulation functions $a_r$, $b_\nu^\star$ and $b_\nu$ defined below,
\begin{align*}
u_{\mathrm{MFE}}(t,x)=& \sum_{r=1}^R a_r(t,x)\, \mathrm{e}^{\iu (\kappa_r \cdot x - \omega_r t)/\eps} \\
&+\eps\lambda \sum_{\nu\in\mathcal{N}} b_\nu^\star(t,x)\, \mathrm{e}^{\iu (\kappa_\nu \cdot x - \omega_\nu^\star t)/\eps}
+\eps\lambda \sum_{\nu\in\mathcal{N}} b_\nu(t,x)\, \mathrm{e}^{\iu (\kappa_\nu \cdot x - \omega_\nu t)/\eps}.
\end{align*}
The functions $a_r(t,x)$ solve the system of advected nonlinear Schr\"odinger equations
\begin{align}\label{eq:a_r}
&\partial_t a_{r} + \kappa_r \cdot\nabla_x a_{r} - \tfrac12{\iu\eps}\Delta_x a_{r} = -\iu\lambda \sum_{\genfrac{}{}{0mm}{}{i,j,k=1,\dots, R:}{\kappa_i-\kappa_j+\kappa_k=\kappa_r}} a_{i} \bar{a}_{j} a_{k}
\\
\nonumber
&\qquad\quad - 2 \iu\eps\lambda^2 \!\!\!\sum_{\genfrac{}{}{0mm}{}{\nu \in \mathcal{N}, p,q=1,\dots, R :}{\kappa_\nu-\kappa_p+\kappa_q=\kappa_r}}
b_\nu  \bar{a}_{p} a_{q}
-\iu \eps\lambda^2 \!\!\!\sum_{\genfrac{}{}{0mm}{}{\nu \in \mathcal{N}, p,q=1,\dots, R :}{\kappa_p-\kappa_\nu+\kappa_q=\kappa_r}}
a_{p} \bar{b}_{\nu} a_q,
\\[1mm] \nonumber
&a_{r}(0,x)=a^0_{r}(x),
\end{align}
where the initial data $a_r^0(x)$ are given for $r=1,\dots,M$ and are set to zero for $r=M+1,\dots,R$.

The functions $b_\nu(t,x)$ for $\nu=(i,j,k)\in \mathcal{N}$ are 
defined by the formula, with $\delta_\nu= \omega_\nu-\omega_\nu^\star\ne 0$,
\begin{equation}\label{eq:b_nu}
b_\nu(t,x) = \frac{1}{\delta_\nu} (a_{i} \bar{a}_{j} a_{k})(t,x) ,
\end{equation}
which turns \eqref{eq:a_r} into a quintic nonlinear Schr\"odinger equation. The function $b_\nu$ is an $\mathcal{O}(\eps)$ approximation to the solution of the advected linear Schr\"odinger equation
$$
\partial_t b_\nu+\kappa_\nu\cdot \nabla_x b_\nu - \tfrac12{\iu\eps}\Delta_x b_\nu
-\frac{\iu\delta_\nu}{\eps}\,b_\nu = -\frac\iu\eps\, a_{i} \bar{a}_{j} a_{k}, \quad b_\nu(0,x)=\frac{1}{\delta_\nu} (a_{i} \bar{a}_{j} a_{k})(0,x) .
$$
The  function $b_\nu^\star$ is 
the solution of the advected linear Schr\"odinger equation
$$
\partial_t b_\nu^\star +\kappa_\nu\cdot \nabla_x b_\nu^\star 
- \tfrac12{\iu\eps}\Delta_x b_\nu^\star = -\iu\lambda \,
2\sum_{r=1}^R |a_r|^2\, b_\nu^\star, 
\ \quad\ b_\nu^\star(0,x)=-b_\nu(0,x).
$$
The modulation functions $a_r$, $b_\nu$ and $b_\nu^\star$ are smooth in the sense that the functions and all their partial derivatives of arbitrary order are bounded independently of $\eps$ and $(t,x)\in [0,T]\times \mathbb{T}^d$ for a time $T>0$ on which a solution of \eqref{eq:a_r} with \eqref{eq:b_nu} exists.
With these functions $a_r$, $b_\nu$ and $b_\nu^\star$, the error $e(t,x)$ is bounded by
$$
\| e \|_{C([0,T] \times \mathbb{T}^d)} \leq C\eps^2.
$$
Both $T$ and $C$ are independent of $\eps\in (0,\eps_0)$ for some $\eps_0>0$,
but they depend on the given wave vectors $\kappa_m$, and $C$ depends  on $T$.
\end{theorem}

\begin{proof}
Let $d(t,x)$ be the defect of the approximate solution $u_{\text{MFE}}(t,x)$ in \eqref{eq:schr}, 
$$
d := \iu\eps\,\partial_t u_{\text{MFE}} + \frac{\eps^2}{2}\Delta u_{\text{MFE}} - \eps\lambda|u_{\text{MFE}}|^2 u_{\text{MFE}}.
$$
Inserting the expression for $u_{\mathrm{MFE}}(t,x)$, we compute
$$
\begin{aligned}
\iu\eps\,\partial_t u_{\text{MFE}} + \frac{\eps^2}{2} \Delta u_{\text{MFE}} 
=  \iu\eps\,\sum_{r=1}^R \left(\partial_t a_r +\kappa_r\,\nabla   a_r - \frac{\iu\eps}{2}\Delta a_r\right)\, \mathrm{e}^{\iu(\kappa_r \cdot x - \omega_r t)/\eps}
\\
+  \iu\eps^2\lambda \,\sum_{\nu\in\mathcal{N}} \left(\partial_t b_\nu + \kappa_\nu\,\nabla   b_\nu - \frac{\iu\eps}{2} \Delta b_\nu -\frac{\iu\delta_\nu}{\eps} b_\nu\right)\, \mathrm{e}^{\iu(\kappa_\nu\cdot  x - \omega_\nu t)/\eps}
\\
+  \iu\eps^2\lambda \,\sum_{\nu\in\mathcal{N}} \left(\partial_t b_\nu^\star + \kappa_\nu\,\nabla   b_\nu^\star - \frac{\iu\eps}{2} \Delta b_\nu^\star \right)\, \mathrm{e}^{\iu(\kappa_\nu\cdot  x - \omega_\nu^\star t)/\eps}.
\end{aligned}
$$
For the nonlinear term we have
\begin{align*}
&|u_{\text{MFE}}|^2 u_{\text{MFE}} = 
\sum_{i,j,k=1}^R a_i \bar{a}_j a_k \e^{\iu(\kappa_{(i,j,k)}\cdot x - \omega_{(i,j,k)}t)/\eps}
\\
&+ \eps\lambda \sum_{\nu\in\mathcal{N}} \sum_{p,q=1}^R \left(
2 b_\nu \bar{a}_p a_q \e^{\iu(\kappa_{(\nu,p,q)}\cdot x - \omega_{(\nu,p,q)}t)/\eps} 
+a_p \bar{b}_\nu a_q \e^{\iu(\kappa_{(p,\nu,q)}\cdot x - \omega_{(p,\nu,q)}t)/\eps} \right)
\\
&+ \eps\lambda \sum_{\nu\in\mathcal{N}} \sum_{p,q=1}^R \left(
2 b_\nu^\star \bar{a}_p a_q \e^{\iu(\kappa_{(\nu,p,q)}\cdot x - (\omega_{(\nu,p,q)}-\delta_\nu)t)/\eps} 
+a_p \bar{b_\nu^\star} a_q \e^{\iu(\kappa_{(p,\nu,q)}\cdot x - (\omega_{(p,\nu,q)}+\delta_\nu)t)/\eps} \right)
\\
&+ \mathcal{O} (\eps^2).
\end{align*}
The equations for $a_r$, $b_\nu$ and $b_\nu^\star$ are constructed such that,
under Assumptions 1 and 2, the defect is of the form
$$
d = \iu\eps \left(\eps \sum_{l=1}^L c_l(t,x) \, \mathrm{e}^{\iu(\wt\kappa_l \cdot x - \wt\omega_l t)/\eps} + \mathcal{O} (\eps^2)\right)
$$
with $\eps$-uniformly bounded and smooth modulation functions $c_l$ and {\it nonresonant} $(\wt\kappa_l,\wt\omega_l)$, i.e.,
$\wt\omega_l\ne \tfrac12 |\wt\kappa_l|^2$. Since all modulation functions are spatially smooth, the same argument as in the proof of
Lemma~\ref{lem:defect-wiener} shows that the above $\mathcal{O}(\eps^2)$-bounds are valid not only in the maximum norm, but also in the stronger norm of the Wiener algebra $A(\mathbb{T}^d)$.

The error $e = u - u_{\text{MFE}}$ then satisfies the equation
$$
\iu\eps\,\partial_t e + \frac{\eps^2}{2} \Delta e = \eps\lambda \left(|u|^2 u - |u_{\text{MFE}}|^2 u_{\text{MFE}} \right) - d.
$$
Applying Duhamel’s principle yields
\begin{align*}
e(t) = \e^{\iu\eps t\Delta/2} e(0)
&- \iu\lambda \int_0^t \e^{\iu\eps(t-s)\Delta/2} \left(|u|^2u - |u_{\mathrm{MFE}}|^2u_{\mathrm{MFE}}\right)\d s\\
&+ \iu\lambda \int_0^t \e^{\iu\eps(t-s)\Delta/2} \,\frac{d(s,\cdot)}{\iu\eps}\,\d s .
\end{align*}
By the partial integration argument of Lemma 5.7 of \cite{carles2010multiphase}, we obtain in the present nonresonant situation
$$
\int_0^t \e^{\iu\eps(t-s)\Delta/2} c_l(s,x) \, \mathrm{e}^{\iu(\wt\kappa_l \cdot x - \wt\omega_l s)/\eps} \,\d s = \mathcal{O}(\eps)
$$
in the Wiener algebra $A(\mathbb{T}^d)$, and so we have
$$
\biggl\|\, \int_0^t \e^{\iu\eps(t-s)\Delta/2} \,\frac{d(s,\cdot)}{\iu\eps}\,\d s 
\,\biggr\|_{A(\mathbb{T}^d)} \le C_0 \eps^2.
$$
Using $e(0)=0$ and trilinear estimates in $A(\mathbb{T}^d)$, we then obtain
$$
\begin{aligned}
\|e(t)\|_{A(\mathbb{T}^d)} 
&\leq C_1 \int_0^t \|e(s)\|_{A(\mathbb{T}^d)} \,\mathrm{d}s + C_0  \eps^2.
\end{aligned}
$$
The stated bound then follows from Gronwall’s inequality.
\end{proof}

 \section{General multiphase problem: numerical method and error bound}\label{sec:multi-scheme}
In view of Theorem~\ref{th:MFE}, the exponentially weighted leapfrog method extends from the particular two-phase case in \eqref{eq:scheme1-two} to the general multiphase situation as follows. 
With the notation of the two previous sections, we approximate
$$
u(t_n,x_j)\approx u^n_j = \sum_{r=1}^R u_{[r],j}^n + 
\sum_{\nu\in\mathcal{N}} w_{[\nu],j}^{\star,n}+\sum_{\nu\in\mathcal{N}} w_{[\nu],j}^{n}
$$
where $u_{[r],j}^n$ is to approximate $u_r(t,x)=a_r(t,x)\, \mathrm{e}^{\iu (\kappa_r\cdot x - \omega_r t)/\eps}$, $w_{[\nu],j}^{\star,n}$ is to approximate $w^\star_\nu(t,x)=b^\star_\nu(t,x)\, \mathrm{e}^{\iu (\kappa_\nu\cdot x - \omega^\star_\nu t)/\eps}$, and $w_{[\nu],j}^n$ is to approximate $w_\nu(t,x)=b_\nu(t,x)\, \mathrm{e}^{\iu (\kappa_\nu\cdot x - \omega_\nu t)/\eps}$
at $t=t_n$ and $x=x_j$. In the following, let $\alpha_r=\omega_r\tau/\eps$, $\beta_r=\kappa_r h/\eps$, $\alpha^\star_\nu=\omega^\star_\nu\tau/\eps$, and $\beta_\nu=\kappa_\nu h/\eps$. We compute $u^{n+1}_j$ as follows for $r=1,\dots,R$ and $\nu=(k,l,m)\in \mathcal{N}$. With the weighted leapfrog finite difference operator of Section~\ref{subsec:wfdm} (or its obvious generalization to higher dimensions $d>1$), we consider the scheme
\[
\begin{aligned}
&\mathcal{D}_{\mathrm{LF}}^{\alpha_r,\beta_r}u_{[r]}\big|^{n}_j = \eps\lambda
\sum_{\genfrac{}{}{0mm}{}{k,l,m=1,\dots, R:}{\kappa_k-\kappa_l+\kappa_m=\kappa_r}} u^n_{[k],j} \bar{u}^n_{{[l],j}} u^n_{{[m],j}}+ 
\chi\eps\lambda
\sum_{\genfrac{}{}{0mm}{}{k,l=1,\dots, R:}{(k,l,r)\in\mathcal{N}}} u^n_{[k],j} \bar{u}^n_{{[l],j}} u^n_{{[r],j}}
\\[1mm]
&\qquad\quad + 2 \eps\lambda \!\!\!\sum_{\genfrac{}{}{0mm}{}{\nu \in \mathcal{N}, p,q=1,\dots, R :}{\kappa_\nu-\kappa_p+\kappa_q=\kappa_r}}
w_{[\nu],j}^n  \bar{u}_{[p],j}^n u_{[q],j}^n
+ \eps\lambda \!\!\!\sum_{\genfrac{}{}{0mm}{}{\nu \in \mathcal{N}, p,q=1,\dots, R :}{\kappa_p-\kappa_\nu+\kappa_q=\kappa_r}}
u_{[p],j}^n \bar{w}_{[\nu],j}^n u_{[q],j}^n
\\[1mm]
&\text{with }\  w^{n}_{[\nu],j} =(1-\chi)\, \frac{\eps\lambda}{\delta_\nu} \, u^{n}_{[k],j} \bar{u}^{n}_{[l],j} u^{n}_{[m],j},
\\[1mm]
&\mathcal{D}_{\mathrm{LF}}^{\alpha^\star_\nu,\beta_\nu}w^\star_{[\nu]}\big|^{n}_j = 2 \eps\lambda \sum_{r=1}^R |u_{[r],j}|^2\, w_{[\nu],j}^{\star,n}, 
\end{aligned}
\]
where $\chi = 1$ if ${h^2\leq c\,\eps^5}$ and zero otherwise. The initial conditions are given by
\[
\begin{aligned}
    u_{[r]}(0,x)&=a^0_{r}(x)\, \mathrm{e}^{\iu \kappa_r\cdot x / \eps}
    \quad\text{for $r=1,\dots,M$ and zero otherwise,}
    \\
    w_{[\nu]}(0,x)&=(1-\chi)\frac{\eps\lambda}{\delta_\nu}u_{[k]}\bar{u}_{[l]}u_{[m]}(0,x)\\
    w^\star_{[\nu]}(0,x)&=-w_{[\nu]}(0,x).
\end{aligned}
\]
An analogous formula holds for the exponentially weighted Crank--Nicolson method; cf.~\eqref{eq:scheme2-two}.

The following theorem provides an $\mathcal{O}(\tau^2 + h^2+\eps^2)$ error bound in the maximum norm for the numerical solutions of the above weighted leapfrog and Crank--Nicolson methods, which extends to an $\eps$-uniform
$\mathcal{O}(\tau^{4/5} + h^{4/5})$ error bound.

\begin{theorem}[Error bound for the weighted leapfrog and Crank-Nicolson methods]\label{th:conv-m}
Under the assumptions of Theorem~\ref{th:MFE}, and in the case of the leapfrog scheme subject to the stability condition $\eps\tau< h^2/\gamma(\beta_\mu)$ for all $\mu\in\{1,\dots,R\}\cup\mathcal{N}$ where $\beta_\mu=\kappa_\mu h/\eps$ and $\gamma$ is defined by \eqref{eq:stability}, the error is bounded by
\[
|u^n_j - u(t_n, x_j)| \leq \min\left(C_0(\tau^2 + h^2 + \eps^2),\; C_1\frac{\tau^2 + h^2}{\eps^3}\right) \leq C(\tau^{4/5} + h^{4/5}),
\]
uniformly for $t_n = n\tau \leq T$ and $x_j = jh$ and $0<\eps\le 1$. The constants $C_0,C_1$ and $C$ are independent of $\eps$, the time step $\tau$, and the mesh size $h$, and independent of $j$ and $n$ for $t_n \leq T$.
\end{theorem}

\begin{proof} Theorem~\ref{th:conv-m} is proved by combining the proofs of Theorems~\ref{th:conv-1} and~\ref{th:MFE}.
We consider two cases based on the relative value of $h^2$ and $\eps^5$:

\begin{enumerate}
    \item $h^2 \ge c\,\eps^5$, i.e., $\chi = 0$. \\
    In this case, proceeding as in the proof of Theorem~\ref{th:conv-1} for each component $u^n_{[r],j}$, $w^{\star,n}_{[\nu],j}$ and $w^n_{[\nu],j}$ of the numerical solution, we obtain that the difference between the numerical solution and the modulated Fourier expansion of the exact solution is bounded by
    \begin{align*}
         |u^{n}_j - u_{\rm MFE}(t_n,x_j)| = 
    \mathcal{O}(\tau^2 + h^2)
    \end{align*}
    uniformly for $j$ and $n$ with $t_n\le T$.
    Together with Theorem~\ref{th:MFE}, this yields the maximum norm error bound
    \[
    |u^{n}_j - u(t_n,x_j)| = \mathcal{O}(\tau^2 + h^2 + \eps^2) 
    \]
    uniformly for $j$ and $n$ with $t_n\le T$.

    \item $h^2 \leq c\,\eps^5$, i.e., $\chi = 1$. \\
    In this regime, $w^\star_{[\nu]}=w_{[\nu]}=0$, the scheme reduces to a standard leapfrog or Crank--Nicolson algorithm. By summing the equations and applying Taylor expansion,  we obtain the error bound
    \[
    |u^{n}_j - u(t_n,x_j)| = \mathcal{O}\left(\frac{\tau^2 + h^2}{\eps^3} \right).
    \]
\end{enumerate}
Combining the two cases yields the stated $\eps$-uniform error bound.
\end{proof}

\color{black}

\bibliographystyle{amsplain}
\bibliography{ref}
\end{document}